\documentclass{amsart}
\usepackage{amsmath}
\usepackage{amssymb}
\usepackage{enumitem}
\usepackage{mathrsfs}
\usepackage{tikz-cd}
\usepackage{bm}
\usepackage{tikz}
\usepackage{xy}
\newcommand{\cta}{\catD} 
\newcommand{\ctb}{\catC} 
\newcommand{\tc}{\mathbb{T}} 
\newcommand{\tp}{\mathbb{P}} 

\newcommand{\cat}[1]{\ensuremath{\mathbf{#1}}}

\newcommand{\catJ}{\cat{J}}
\newcommand{\HilbP}{\Hilb_\sim}
\newcommand{\FHilbP}{\FHilb_\sim}

\newcommand{\MatS}{\Mat_S}

\newcommand{\VecSp}{\cat{Vec}}
\newcommand{\FVecSp}{\cat{FVec}}
\newcommand{\GSet}{G\text{-}\cat{Set}}
\newcommand{\VecP}{\cat{Vec}_{\sim}}
\newcommand{\FVecP}{\cat{FVec}_{\sim}}
\newcommand{\VecProj}{\cat{Proj}_k} 
\newcommand{\Ring}{\cat{Rng}}
\newcommand{\catC}{\cat{C}}
\newcommand{\catD}{\cat{D}}
\newcommand{\Mat}{\cat{Mat}}
\newcommand{\FHilb}{\cat{FHilb}}
\newcommand{\Hilb}{\cat{Hilb}}

\newcommand{\Grp}{\cat{Grp}}
\newcommand{\id}[1]{\ensuremath{\mathrm{id}_{#1}}}
\newcommand{\Aut}{\mathrm{Aut}} 
\newcommand{\op}{\ensuremath{\mathrm{\rm op}}}
\newcommand{\isomto}{\xrightarrow{\sim}}
\newcommand{\coproj}{\kappa}
\newcommand{\biprod}{\oplus}
\newcommand{\pinit}{0}
\newcommand{\pcoprod}{\mathbin{\dot{+}}}
\newcommand{\pprod}{\mathbin{\dot{\times}}}
\newcommand{\pbiprod}{\mathbin{\dot{\oplus}}}

\newcommand{\pcoproj}{\coproj}
\newcommand{\pproj}{\pi}
\newcommand{\quotP}{\mathbb{P}} 
\newcommand{\quot}[2]{{#1} \!/\! {#2}}
\newcommand{\Gen}{I} 
\newcommand{\TrI}{\mathbb{T}} 
\newcommand{\MD}{\cat{Co}} 
\newcommand{\BMD}{\cat{BCo}} 
\newcommand{\MDG}{\cat{Co}_{\mathsf{G}}} 
\newcommand{\BMDG}{\cat{BCo}_{\mathsf{G}}} 
\newcommand{\MDPh}{\cat{PhCo}} 
\newcommand{\BMDPh}{\cat{BPhCo}} 
\newcommand{\deff}[1]{\emph{#1}}
\newcommand{\indef}[1]{\emph{#1}}
\newcommand{\GP}{\mathsf{GP}}
\newcommand{\GPa}[1]{\mathrm{GP}({#1})} 
\newcommand{\plusI}[1]{\mathsf{GP}({#1})}
\newcommand{\plusIdag}[1]{\mathsf{GP}^{\dagger}({#1})}
\newcommand{\tens}{\widehat{\otimes}}
\newcommand{\dtens}{\tens} 
\newcommand{\obb}[1]{{\bm{#1}}}
\newcommand{\aalpha}{\obb{\alpha}}
\newcommand{\llambda}{\obb{\lambda}}
\newcommand{\ssigma}{\obb{\sigma}}
\newcommand{\rrho}{\obb{\rho}}

\usetikzlibrary{decorations.pathreplacing,decorations.markings,arrows.meta,backgrounds}
\pgfdeclarelayer{edgelayer}
\pgfdeclarelayer{nodelayer}
\pgfsetlayers{background,edgelayer,nodelayer,main}

\tikzstyle{cdot}=[circle, draw=black, fill=black!25, inner sep=.4ex] 

\tikzstyle{whitedot}=[circle, draw=black, fill=white, inner sep=.4ex]

\newcommand{\tinycup}{\smash{\raisebox{2pt}{\hspace{-2pt}\ensuremath{\begin{pic}[scale=0.2]
   \pgftransformscale{1.5} \draw[arrow=.6, scale = 1] (0,0) to[out=-90,in=-90,looseness=1.5] (1.5,0);
\end{pic}}}}}

\newcommand{\tinycap}{\smash{\raisebox{-3pt}{\hspace{-2pt}\ensuremath{\begin{pic}[scale=0.2, yscale=-1]
   \pgftransformscale{1.5} \draw[arrow=.6, scale = 1] (0,0) to[out=-90,in=-90,looseness=1.5] (1.5,0);
\end{pic}}}}}

\newcommand{\tinymultflip}[1][cdot]{
\smash{\raisebox{-2pt}{\hspace{-5pt}\ensuremath{\begin{pic}[scale=0.4,yscale=1]
    \node (0) at (0,0) {};
    \node[#1, inner sep=1.5pt] (1) at (0,0.55) {};
    \node (2) at (-0.5,1) {};
    \node (3) at (0.5,1) {};
    \draw (0.center) to (1.center);
    \draw (1.center) to [out=left, in=down, out looseness=1.5] (2.center);
    \draw (1.center) to [out=right, in=down, out looseness=1.5] (3.center);
    \node[#1, inner sep=1.5pt] (1) at (0,0.55) {};
\end{pic}
}\hspace{-3pt}}}}

\newcommand{\tinycounit}[1][cdot]{
\smash{\raisebox{-5pt}{\hspace{-3pt}\ensuremath{\begin{pic}[scale=0.4,yscale=1]
    \node (0) at (0,0) {};
    \node[#1, inner sep=1.5pt] (1) at (0,0.55) {};
    \draw (0.center) to (1.north);
        \node[#1, inner sep=1.5pt] (1) at (0,0.55) {};
\end{pic}
}\hspace{-1pt}}}}
\usetikzlibrary{decorations.pathreplacing,decorations.markings,arrows.meta,backgrounds,shapes}
\pgfsetlayers{background,edgelayer,nodelayer,main}
\tikzstyle{none}=[inner sep=0mm]
\tikzstyle{every loop}=[]
\tikzstyle{mark coordinate}=[inner sep=0pt,outer sep=0pt,minimum size=3pt,fill=black,circle]
\tikzset{arrow/.style={decoration={
    markings,
    mark=at position #1 with \arrow{>[length=2pt, width=3pt]}},
    postaction=decorate},
    reverse arrow/.style={decoration={
    markings,
    mark=at position #1 with {{\arrow{<[length=2pt, width=3pt]}}}},
    postaction=decorate}
}
\tikzstyle{box}=[draw,shape=rectangle,inner sep=2pt,minimum height=5mm,minimum width=6mm,fill=white] 
\tikzstyle{medium box}=[draw,shape=rectangle,inner sep=2pt,minimum height=5mm,minimum width=10mm,fill=white] 
\tikzstyle{dot}=[inner sep=0mm,minimum width=2mm,minimum height=2mm,draw,shape=circle]  
\tikzstyle{black dot}=[dot,fill=black]
\tikzstyle{white dot}=[dot,fill=white,,text depth=-0.2mm]
\tikzstyle{grey dot}=[dot,fill=black!25] 
\tikzstyle{corner1}=[box,fill=white, font=\footnotesize] %
\tikzstyle{corner2}=[dot,fill=white, font=\footnotesize] %
\tikzstyle{corner3}=[dot,fill=black!25, font=\footnotesize] %
\tikzstyle{corner4}=[dot,fill=black, font=\footnotesize] %
\tikzstyle{scalar}=[circle,draw,inner sep=0.5pt,font=\small] 
\tikzstyle{point}=[regular polygon,regular polygon sides=3,draw,scale=0.75,inner sep=-0.5pt,minimum width=9mm,fill=white,regular polygon rotate=180]\tikzstyle{copoint}=[regular polygon,regular polygon sides=3,draw,scale=0.75,inner sep=-0.5pt,minimum width=9mm,fill=white]
\tikzstyle{wide copoint}=[fill=white,draw,shape=isosceles triangle,shape border rotate=90,isosceles triangle stretches=true,inner sep=0pt,minimum width=1.5cm,minimum height=6.12mm]
\tikzstyle{wide point}=[fill=white,draw,shape=isosceles triangle,shape border rotate=-90,isosceles triangle stretches=true,inner sep=0pt,minimum width=1.5cm,minimum height=6.12mm,yshift=-0.0mm]
\tikzstyle{every picture}=[baseline=-0.25em,scale=0.5]
\tikzstyle{label}=[font=\footnotesize,text height=1ex, text depth=0.15ex]
\newenvironment{pic}[1][] {\begin{aligned}\begin{tikzpicture}[scale=2.0, font=\tiny,#1]}{\end{tikzpicture}\end{aligned}} 

\usepackage[foot]{amsaddr}
\theoremstyle{plain}
\newtheorem{theorem}{Theorem}[section]
\newtheorem{proposition}[theorem]{Proposition}
\newtheorem{corollary}[theorem]{Corollary}
\newtheorem{lemma}[theorem]{Lemma}

\theoremstyle{definition}
\newtheorem{definition}[theorem]{Definition}
\newtheorem{example}[theorem]{Example}
\newtheorem{examples}[theorem]{Examples}
\newtheorem{remark}[theorem]{Remark}

\title{Quotient Categories and Phases}
\author{Sean Tull}
\address{Department of Computer Science, University of Oxford 
\\ 
15 Parks Rd, Oxford, United Kingdom OX1 3QD}
\email{sean.tull@cs.ox.ac.uk}
\thanks{This work forms a part of the author's DPhil thesis, and was supported by EPSRC Studentship OUCL/2014/SET. We thank Chris Heunen and Dan Marsden for useful feedback and Martti Karvonen for a helpful discussion on isotropy.}

\begin{document}

\begin{abstract}
We study properties of a category after quotienting out a suitable chosen group of isomorphisms on each object. Coproducts in the original category are described in its quotient by our new weaker notion of a `phased coproduct'. We examine these and show that any suitable category with them arises as such a quotient of a category with coproducts. Motivation comes from projective geometry, and also quantum theory where they describe superpositions in the category of Hilbert spaces and continuous linear maps up to global phase. The quotients we consider also generalise those induced by categorical isotropy in the sense of Funk et al.
 
\end{abstract}

\maketitle

\section{Introduction}

In many mathematical situations one considers objects and maps only up to equivalence under some `trivial' isomorphisms. In doing so we pass from our original category to a quotient category $\catC \!/\!\! \sim$. Examples include projective geometry~\cite{coxeter2003projective} and quantum theory~\cite{abramskycoecke:categoricalsemantics}, in which vectors equal by a suitable constant factor are identified, as well as the more recent setting of quotienting out elements of \emph{isotropy} in a (small) category \cite{funk2012isotropy,funk2018higher}.

When our original category $\catC$ comes with desirable features such as (co)limits, it is  natural to ask what remains of these in its quotient category. Conversely, what properties of $\catC \!  /  \! \! \sim$ ensure that it arises as such a quotient, and can we use them to recover $\catC$?

In this article we address these questions for categories $\catC$ with coproducts. To do so we describe what remains of these coproducts in the quotient category $\catC \!  /  \! \! \sim$ using our new more general notion of a \emph{phased coproduct}. Roughly, these are coproducts whose induced morphisms $A + B \to C$ are unique only up to some coprojection-preserving isomorphism on $A + B$, which we call a \emph{phase}. Phased coproducts generalise to phased (co)limits, an elementary notion of weak (co)limit in a category which does not seem to have appeared elsewhere. Despite their generality, we show phased coproducts to be well-behaved, for example satisfying their own forms of uniqueness and associativity laws. 

 More surprisingly still, we establish a result providing a converse to the above situation; given any suitable category $\catC$ with phased coproducts, we construct a new category $\plusI{\catC}$ with coproducts from which it arises as such a quotient 
 $\catC \simeq \plusI{\catC} \!/\!\! \sim$.

 The $\GP$ construction is most natural when $\catC$ has \emph{phased biproducts}, generalising standard categorical biproducts, and comes with a suitably compatible monoidal structure $(\otimes, I)$. In particular it yields a precise correspondence between such monoidal categories and those with standard (co,bi)products along with a chosen group of \emph{global phase} scalars $\mathbb{P} \subseteq \Aut(I)$, which we quotient out to obtain $\catC$. 

These results have particular application to quantum theory, in which one often studies the quotient $\Hilb_\sim$ of the category $\Hilb$ of Hilbert spaces and continuous linear maps after identifying all global phases, here meaning scalars $z \in \mathbb{C}$ with $|z| = 1$. The category $\Hilb$ has nicer mathematical features, including biproducts which describe quantum superpositions, but only $\Hilb_\sim$ has a direct physical interpretation. Phased coproducts allow us to describe superpositions in the latter category, and to recover the former as 
\[
\Hilb \simeq \plusI{\Hilb_\sim}
\]
This last fact is useful for so-called \emph{reconstructions} of quantum theory~\cite{catreconstruction}.

In its most general form, our approach applies to any category $\catC$ with coproducts from which we quotient out a chosen subgroup of \emph{trivial isomorphisms} $\subseteq \Aut(A)$ on each object $A$, which must satisfy a simple lifting property. Such automorphisms may in fact be seen as a weakening of the notion of elements of isotropy in a small category, which come with well-behaved \emph{choices} of such liftings (though our results were obtained without this knowledge, for which we thank Martti Karvonen). Isotropy was originally introduced for Grothendieck toposes by Funk et al.~\cite{funk2012isotropy}, before being extended to small categories \cite{khan2017aspects,funk2018higher}, and algebraic theories \cite{hofstra2018isotropy}. Quotients under isotropy elements have also been studied extensively \cite{funk2018higher}. 

\subsection*{Future work}

In future it would be desirable to explore the significance of our results for isotropy theory. It remains to find a (small) category whose quotient under isotropy has non-trivial phased coproducts, or prove that none exists. If such quotients exist one may hope to characterise them via properties of their phases. The general $\GP$ construction, and its requirement of a \emph{phase generator} object, would also both be interesting to understand in the isotropy context. 

More broadly, phased (co)limits provide a new elementary categorical definition whose properties and connections with other notion of weak (2-)limit remain to be explored. Finally,  our quotients, which turn certain isomorphisms into identities, should be compared with \emph{localisations} in which certain morphisms are turned into isomorphisms~\cite{hirschhorn2009model}.  

\subsection*{Structure of article}

We begin by introducing phased coproducts and their properties in Section~\ref{sec:ph-coprod}, before meeting examples and our most general notion of quotient category in Section~\ref{sec:examplesinquotientcategories}. We describe the most general form of the $\GP$ construction in Section~\ref{sec:phToCoprod}, which requires the existence of a phase generator. However these technicalities fall away in the setting of monoidal categories which we treat in Section~\ref{sec:monoidal}. 

Following this, we study the construction in the context of other features of our motivating examples, by considering biproducts in Section~\ref{sec:phbiprod}, compact closure in Section~\ref{sec:compact_cats} and finally dagger categories in Section~\ref{sec:daggers}. Appendix \ref{app:universality} shows how the $\GP$ construction can be made functorial and given a universal property. 


\section{Phased Coproducts} \label{sec:ph-coprod}

Our central definition in this article is the following. 

\begin{definition} \label{def:ph_coprod} \label{not:phcoprod} \label{not:phcoproj}
In any category, a \deff{phased coproduct} of a pair of objects $A, B$ is an object $A \pcoprod B$ together with a pair of morphisms  $\pcoproj_A \colon A \to A \pcoprod B$ and $\pcoproj_B \colon {B \to A \pcoprod B}$ satisfying the following. Firstly, for any pair of morphisms $f \colon A \to C$, $g \colon B \to C$, there exists $h \colon {A \pcoprod B \to C}$ making the following commute:
\[
\begin{tikzcd}
A \rar{\coproj_A} \drar[swap]{f} & A \pcoprod B 
\dar[dashed]{h}
& B \lar[swap]{\coproj_B} \dlar{g} \\ 
& C &
\end{tikzcd} 
\]
Secondly, any pair of such morphisms $h,h'$ have that $h' = h \circ U$
\[
\begin{tikzcd}
 \arrow[loop left, "U",distance = 2em] A \pcoprod B \rar[shift right = 2, swap]{h'} \rar[shift left = 2]{h}   & C
\end{tikzcd}
\]
for some endomorphism $U$ of $A \pcoprod B$ which satisfies
\begin{align} \label{eq:phase-eq}
U \circ \pcoproj_A = \pcoproj_A
\qquad
U \circ \pcoproj_B = \pcoproj_B
\end{align}
We call any endomorphism $U$ of $A \pcoprod B$ satisfying \eqref{eq:phase-eq} a \deff{phase} for $A \pcoprod B$, and the morphisms $\pcoproj_A$, $\pcoproj_B$ \deff{coprojections}.
\end{definition}

The typical notion of a categorical coproduct $A + B$ is precisely a phased coproduct whose only phase is the identity morphism.
Straightforwardly extending the above, a \indef{phased coproduct} of any collection of objects $(A_i)_{i \in I}$ is defined as an object $A$ together with morphisms 
$(\coproj_i \colon A_i \to A)_{i \in I}$ satisfying the following.
 Firstly, for any collection of morphisms 
\[
\begin{tikzcd}
A_i \rar{f_i} & B
\end{tikzcd}
\]
 there exists $f \colon A \to B$ with $f \circ \coproj_i = f_i$ for all $i$. Furthermore, any such $f, f' \colon A \to B$ have $f' = f \circ U$ for some $U \colon A \to A$ satisfying $U \circ \pcoproj_i = \pcoproj_i$ for all $i$, which we call a \indef{phase}. A phased coproduct of finitely many objects $A_1, \dots, A_n$ is denoted $A_1 \pcoprod \dots \pcoprod A_n$.

Despite their generality, phased coproducts are surprisingly well-behaved. In particular they are unique up to (non-unique) isomorphism.

\begin{lemma} \label{lem:isoms}
Let $A$ and $B$ be phased coproducts of objects $(A_i)_{i \in I}$ with respective coprojections $\coproj_i \colon A_i \to A$ and $\mu_i \colon A_i \to B$ for $i\in I$. Then any morphism $f$ for which each diagram
\[
\begin{tikzcd}[row sep = small, column sep = small]
& A_i \arrow[dl, "\coproj_i", swap] \arrow[dr, "\mu_i"] & \\ 
A \arrow[rr,"f",swap] & & B 
\end{tikzcd}
\]
commutes is an isomorphism.
Conversely, any object $C$ with an isomorphism $g \colon {A \isomto C}$ forms a phased coproduct of the $A_i$ with coprojections $\nu_i := g \circ \coproj_i$.
\end{lemma}

\begin{proof}
For the first statement, let $g \colon B \to A$ with $g \circ \mu_i = \coproj_i$ for all $i$. Then $f \circ g$ preserves the $\mu_i$ and so there is some phase $U$ on $B$ with $f \circ g \circ U = \id{B}$. But then $g \circ U \circ f$ preserves the $\pcoproj_i$ and so there is a phase $V$ on $A$ with 
\[
g \circ U \circ f \circ V = \id{A}
\]
Hence $g \circ U$ has left and right inverses, making it and hence $f$ both isomorphisms.

For the second statement, given any tuple $(f_i \colon A_i \to D)^n_{i=1}$, let $f \colon A \to D$ satisfy $f \circ \pcoproj_i = f_i$ for all $i$. Then $f \circ g^{-1} \circ \nu_i = f_i$ for all $i$. Moreover, if $h \circ \nu_i = k \circ \pcoproj'_i$ for all $i$ then 
\[
h \circ g^{-1} \circ \pcoproj_i = k \circ g^{-1} \circ \pcoproj_i
\] 
for all $i$ and so for some phase $U$ on $A$ we have that $h = k \circ V$ where $V = g^{-1} \circ U \circ g$. Finally, $V$ is easily seen to preserve the $\nu_i$.
\end{proof}

\begin{corollary} \label{phase_iso}
Any phase of a phased coproduct is an isomorphism. 
\end{corollary}

Next we observe that phased coproducts are associative in a suitable sense.

\begin{proposition}(Associativity)\label{prop:assoc} 
For any phased coproduct $A \pcoprod B$, any phased coproduct ${(A \pcoprod B) \pcoprod C}$ forms a phased coproduct of $A, B, C$ with coprojections:
\[
\begin{tikzcd}[row sep = tiny]
A \arrow[dr,"\pcoproj_A",pos=0.2]& & \\ 
B \arrow[r, "\pcoproj_B"] & A \pcoprod B \arrow[r,"\pcoproj_{A\pcoprod B}"]& (A \pcoprod B) \pcoprod C\\ 
C \arrow[urr,"\pcoproj_C",swap] & &
\end{tikzcd}
\]
More generally $((A_1 \pcoprod A_2) \pcoprod \dots  ) \pcoprod A_n$ forms a phased coproduct $A_1 \pcoprod \dots \pcoprod A_n$.
\end{proposition}
\begin{proof}
We prove the first case, with the $n$-ary case being similar. 

For any morphisms $f, g, h$ from $A, B, C$ to $D$ respectively, let $k \colon A \pcoprod B \to D$ satisfy $k \circ \pcoproj_A = f$ and $k \circ \pcoproj_B = g$. Then any morphism
\[
\begin{tikzcd}
(A \pcoprod B) \pcoprod C
\rar{t}
& D 
\end{tikzcd}
\]
with $t \circ \pcoproj_{A \pcoprod B} = k$ and  $t \circ \pcoproj_C = h$ composes with the morphisms above to give $f, g, h$ respectively. For uniqueness, suppose that $t'$ is another such morphism. Then there is a phase $U$ on $A \pcoprod B$ with $t' \circ \pcoproj_{A \pcoprod B} = t \circ \pcoproj_{A \pcoprod B} \circ U$. Now let $V$ be an endomorphism of $(A \pcoprod B) \pcoprod C$ with 
\[
V \circ \pcoproj_{A \pcoprod B} = \pcoproj_{A \pcoprod B} \circ U \qquad V \circ \pcoproj_C = \pcoproj_C
\]
Then immediately we have $t \circ V \circ \pcoproj_{A \pcoprod B} = t' \circ \pcoproj_{A \pcoprod B}$ and $t \circ V \circ \pcoproj_C = t' \circ \pcoproj_C$. So there is some $W$ preserving $\pcoproj_{A \pcoprod B}$ and $\pcoproj_C$ with $t = (t \circ V) \circ W$. Finally $V \circ W$ preserves each of the proposed coprojections as required.
\end{proof}

Let us now consider a phased coproduct of an empty collection of objects, which by definition is precisely the following. In any category, a \indef{phased initial object} is an object $\pinit$ for which every object $A$ has a morphism $\pinit \to A$, and such that for any pair of morphisms $a, b \colon \pinit \to A$ there is an endomorphism $U$ of $\pinit$ with $b = a \circ U$. In fact this notion typically coincides with a familiar one. 

\begin{proposition} \label{prop:phase_init_unit}
In a category with binary phased coproducts, any phased initial object $0$ is an initial object and each coprojection $\pcoproj_A \colon A \to A \pcoprod \pinit$ is an isomorphism.
\end{proposition}
\begin{proof}
We first show that $\pcoproj_A$ is an isomorphism. 
Let 
\[
\begin{tikzcd}
A \pcoprod \pinit \rar{f} & A
\end{tikzcd}
\]
with $f \circ \pcoproj_A = \id{A}$ and $f \circ \pcoproj_{\pinit}$ being any morphism $\pinit \to A$. Then $f$ makes $\pcoproj_A$ split monic. 
Because $\pinit$ is phased initial, it has an endomorphism $z$ with  
$\pcoproj_A \circ f \circ \pcoproj_0 = \pcoproj_0 \circ z$, which is an isomorphism by Lemma~\ref{lem:isoms}. 
Next let $g$ be an endomorphism of $A \pcoprod \pinit$ with $g \circ \pcoproj_A = \pcoproj_A$ and $g \circ \pcoproj_{\pinit} = \pcoproj_{\pinit} \circ z^{-1}$. Then it may be readily verified that, by construction, $U := \pcoproj_A \circ f \circ g$ preserves $\pcoproj_0$ and $\pcoproj_A$. Hence $U$ is a phase and so an isomorphism, making $\pcoproj_A$ split epic and hence an isomorphism also. 

We now show that $\pinit$ is initial. Given $a, b \colon \pinit \to A$ let $g, h \colon A \pcoprod \pinit \to A$ with
\[
g \circ \pcoproj_\pinit = a \qquad h \circ \pcoproj_\pinit = b \qquad g \circ \pcoproj_A = \id{A} = h \circ \pcoproj_A
\]
Then $g = \pcoproj_A^{-1} = h$ and so $a = g \circ \pcoproj_\pinit = h \circ \pcoproj_\pinit = b$. 
\end{proof}

\begin{corollary} \label{all_phase_coprod} 
A category has phased coproducts of all finite collections of objects precisely when it has binary phased coproducts and an initial object.
\end{corollary}

Thanks to this result, we will often only need to refer to binary phased coproducts in what follows.

\begin{remark}(Phased Limits) \label{rem:phasedproducts}
We may have defined phased products $A \pprod B$\label{not:phproduct} and phased terminal objects by dualising the above definitions, but coproducts will be more natural for our later applications to monoidal categories. 

 In fact one can straightforwardly define a general theory of \emph{phased (co)limits}, by saying that a diagram $D \colon \catJ \to \catC$ has a phased (co)limit in $\catC$ if the category of (co)cones over $D$ has a phased terminal (resp.~initial) object. However we will not pursue more general phased limits here. 
\end{remark}


\section{Examples in Quotient Categories} \label{sec:examplesinquotientcategories} 

Let us now see how phased coproducts arise naturally. 

\begin{definition}
By a choice of \deff{trivial isomorphisms} on a category $\cta$ we mean a choice, for each object $A$, of a subgroup $\TrI_A$\label{not:trivisom} of the group of isomorphisms $A \isomto A$ such that for all $f \colon A \to B$ and $q \in \TrI_B$ there exists $p \in \TrI_A$ making the following commute: 
\begin{equation} \label{eq:triv-isom}
\begin{tikzcd}
A \rar{f} \dar[dashed,swap]{p} & B \dar{q} \\
A \rar[swap]{f} & B
\end{tikzcd}
\end{equation}
We call a choice of trivial isomorphisms \indef{transitive}  when, conversely, for all such morphisms $f$ and every $p \in \TrI_A$ we have $f \circ p = q \circ f$ for some $q \in \TrI_B$. With or without transitivity, such a choice defines a congruence $\sim$ on $\cta$ given on morphisms $f, g \colon A \to B$ by 
\begin{equation} \label{eq:triv-isom-cong}
f \sim g \quad \text{ if } \quad  f = g \circ p \text{ for some $p \in \TrI_A$}
\end{equation}
In fact this congruence suffices to recover $\TrI_A$ as $\{f \colon A \to A \mid f \sim \id{A}\}$, and so we often equate a choice of trivial isomorphisms with its congruence. 

We write $\quot{\cta}{\sim}$\label{not:quotcat} for the category whose morphisms are equivalence classes $[f]_\tc$ of morphisms $f$ in $\cta$ under $\sim$. There is a wide full functor $[-]_\tc \colon \cta \to \quot{\cta}{\sim}$ given by taking equivalence classes.
\end{definition}

\begin{lemma} \label{lem:generalrecipeforphcoprod}
Let $\cta$ be a category with finite coproducts and a choice of trivial isomorphisms. Then $\quot{\cta}{\sim}$ has finite phased coproducts. Moreover $[-]_\tc$ sends coproducts in $\cta$ to phased coproducts in $\quot{\cta}{\sim}$. 
\end{lemma} 


\begin{proof}
Any initial object in $\cta$ is initial in $\quot{\cta}{\sim}$. For any $[f]_\tc \colon A \to C, [g]_\tc \colon B \to C$, the morphism $h \colon A + B \to C$ with $h \circ \coproj_A = f$ and $h \circ \coproj_C =g$ certainly has $[h]_\tc \circ [\pcoproj_A]_\tc = [f]_\tc$ and $[h]_\tc \circ [\pcoproj_B]_\tc = [g]_\tc$. Given any other such $[h']_\tc$, we have $h' \circ \pcoproj_A = f \circ p$ and $h \circ \pcoproj_B = g \circ q$ for some $p \in \TrI_A$ and $q \in \TrI_B$. Then $h = h' \circ U$ where $U \circ \pcoproj_A = \pcoproj_A \circ p$ and $U \circ \pcoproj_B = \pcoproj_B \circ q$, with $[U]_\tc$ preserving the $[\pcoproj_A]_\tc$ and $[\pcoproj_B]_\tc$ in $\quot{\cta}{\sim}$. 
\end{proof}

Our motivating examples are the following.

\begin{example} \label{example:vecspac}
Let $\VecSp_{}$ be the category of complex vector spaces and linear maps, and choose as trivial isomorphisms on $V$ all linear maps $e^{i \theta} \cdot \id{V}$ for $\theta \in [0,2\pi)$. Then the induced congruence is 
\begin{equation} \label{eq:HilbRelation}
f \sim g \quad \text{   if   } \quad  f = e^{i \theta} \cdot g
\end{equation}
for some such $\theta$. $\VecSp_{}$ has coproducts given by direct sum $V \biprod W$ of vector spaces, together with the obvious inclusions $\coproj_V \colon V \to V \biprod W$ and $\coproj_W \colon W \to V \biprod W$, and so these become phased coproducts in $\VecP := \quot{\VecSp}{\sim}$. Phases on $V \pbiprod W$ in this category are precisely $\sim$-equivalence classes of linear automorphisms of the form
\[
U = 
  \begin{pmatrix}
    \id{V} & 0 \\
    0 & e^{i \theta} \cdot \id{W} 
  \end{pmatrix}
\]
\end{example}

\begin{example} \label{example:hilbP}
The above example is relevant to quantum theory after restricting from $\VecP$ to $\HilbP$, the quotient of the category $\Hilb$ of Hilbert spaces and continuous linear maps by the above equivalence relation. Again $\Hilb$ has coproducts given by direct sum of Hilbert spaces, and these form phased coproducts in $\HilbP$ with phases of the above form. 
\end{example}

\begin{example} \label{example:proj-geom}
Let $\FVecSp_{k}$ be the category of finite-dimensional $k$-vector spaces, for some field $k$. Now choose as trivial isomorphisms all maps of the form $\lambda \cdot \id{V}$ for any non-zero $\lambda$, and let $\VecProj := \quot{\FVecSp_{k}}{\sim}$. Morphisms here are linear maps up to an overall scalar $\lambda$. Identifying vectors $\psi \colon k \to V$ with the same span in this way leads to \emph{projective geometry}~\cite{coxeter2003projective}. In particular automorphisms of $k^n$ in $\VecProj$ are precisely \emph{homographies} on the projective space $\mathsf{PG}(n,k)$. Note however that $\VecProj$ differs from usual projective geometry by including zeroes and non-injective maps. 
\end{example}

\begin{example}
Let $G$ be an abelian group and let $\GSet$ be the category of sets $A$ equipped with a group action $a \mapsto g \cdot a$, with morphisms being maps $f \colon A \to B$ which are equivariant i.e.~satisfy $f(g \cdot a) = g \cdot f(a)$ for all $g \in G, a \in A$. Choose as trivial isomorphisms on $A$ all maps of the form $g \cdot (-) \colon A \to A$ for some $g \in G$. Then the quotient $\quot{\GSet}{\sim}$ consists of equivalence classes of maps under $f \sim f'$ whenever there is some $g \in G$ with $f(a) = f'(g \cdot a)$ for all $a \in A$. It has finite phased coproducts given by coproducts in $\GSet$, i.e.~disjoint union.
\end{example}

\begin{example}(Isotropy)
It is natural to strengthen our definition of trivial isomorphism $\TrI$ to require that each $q \in \TrI_B$ comes with a \emph{choice}, for each morphism $f \colon A \to B$, of lifting $p = q_f \in \TrI_A$ making \eqref{eq:triv-isom} commute. To be well-behaved this choice should further satisfy $g \circ q_{f \circ g} = q_f \circ g$ for all $g \colon C \to A$.

In fact this coincides with an established notion due to Funk et al.~\cite{funk2012isotropy,funk2018higher}. 
In any small category $\cta$, an automorphism $q$ of $B$ along with choices of such automorphisms $q_f$ is called an \emph{element of isotropy} at $B$. More formally, it is a natural automorphism of the forgetful functor $\quot{\cta}{B} \to \cta$. The collections of isotropy elements $Z(B)$ at each object $B$ extend to a presheaf of groups $Z(\cta) \colon \cta^\op \to \Grp$ called the \emph{isotropy group} of $\cta$. 

Thus any (small) category $\cta$ comes with a canonical choice of trivial isomorphisms $\TrI_A=Z(A)$ provided by its isotropy group. Quotienting out by the congruence $\sim$ is a major aspect of isotropy theory \cite{funk2018higher}. When $\cta$ has coproducts this provides a source of examples of categories $\quot{\cta}{\sim}$ with phased coproducts. 

We note however that in many cases the phases in $\quot{\catD}{\sim}$ are in fact trivial, giving this category plain coproducts. This owes to the fact that in $\catD$ for all $p \in Z(A)$, $q \in Z(B)$, the automorphism $p + q$ belongs to $Z(A + B)$, but the classes $[p + q]$ of such morphisms are always precisely the phases in $\quot{\catD}{\sim}$.  This property holds, for example, of isotropy in the (opposite) categories $\Grp^\op$, $\Ring^\op$ of (finitely presentable) groups or rings, as well as the other examples of algebraic theories considered in \cite{hofstra2018isotropy}. We leave open the problem of determining a quotient category under isotropy with non-trivial phased coproducts, or proving that none exists. 
\end{example}

Each of our first three examples of trivial isomorphisms are transitive, giving their induced phased coproducts a property which will be useful in what follows. First, let us say that a morphism $f \colon A \pcoprod B \to C \pcoprod D$ is \indef{diagonal} when $f \circ \pcoproj_A = \pcoproj_C \circ g$ and $f \circ \pcoproj_B = \pcoproj_D \circ h$ for some $g, h$. 

\begin{definition} \label{def:trans_phases}
A category with phased coproducts has \deff{transitive phases}  when every diagonal morphism $f \colon A \pcoprod B \to C \pcoprod D$ and phase $U$ of $A \pcoprod B$ has 
\[
f \circ U = V \circ f
\]
for some phase $V$ of $C \pcoprod D$.
\end{definition}

\section{From Phased Coproducts to Coproducts} \label{sec:phToCoprod}

We now wish to find a converse construction to Lemma~\ref{lem:generalrecipeforphcoprod}, allowing us to exhibit any suitable category with phased coproducts as a quotient of one with coproducts.

\begin{definition} \label{def:justplusI}
Let $\ctb$ be a category with finite phased coproducts and a distinguished object $\Gen$. The category $\plusI{\ctb}$ is defined as follows:
\begin{itemize}
\item 
objects are phased coproducts of the form $\obb{A} = A \pcoprod \Gen$ in $\ctb$ (each including as data the objects $A, \Gen$ and morphisms $\coproj_A$, $\coproj_\Gen$);
\item 
morphisms $f \colon \obb{A} \to \obb{B}$ are diagonal morphisms in $\ctb$ with $f \circ \pcoproj_\Gen = \pcoproj_\Gen$.
\[
\begin{tikzcd}
\obb{A} \rar{f} & \obb{B} \\ 
A \uar{\pcoproj_A} \rar[dashed,swap]{\exists} & B \uar[swap]{\pcoproj_B}
\end{tikzcd}
\qquad
\begin{tikzcd}
\obb{A} \arrow[rr, "f"] & & \obb{B} \\ 
& I \arrow[ul,"\pcoproj_\Gen"] \arrow[ur,swap,"\pcoproj_\Gen"]& 
\end{tikzcd}
\]
\end{itemize}
\end{definition}
Such diagonal morphisms are straightforwardly checked to be closed under composition, making  $\mathsf{GP}(\ctb)$ a well-defined category with composition and identity morphisms being the same as in $\ctb$. Our notation $\mathsf{GP}$ stands for `global phases', see Section~\ref{sec:monoidal}. 

A sufficient condition on $I$ for $\plusI{\catC}$ to have coproducts is the following. Call a morphism $f \colon A \pcoprod B \to C$ \emph{phase monic} when $f \circ U = f \circ V \implies U = V$ for all phases $U, V$. Similarly a morphism $g \colon C \to A \pcoprod B$ is \emph{phase epic} when $U \circ g = V \circ g \implies U = V$ for phases $U, V$.

\begin{definition} \label{def:phase-gen}
Let $\catC$ be a category with finite phased coproducts. We say an object $\Gen$ is a \emph{phase generator} when:
\begin{itemize}
\item 
any $\triangledown \colon I \pcoprod I \to I$ with $\triangledown \circ \pcoproj_1 = \id{I} = \triangledown \circ \pcoproj_2$ is phase monic;
\item 
any diagonal monomorphism $m \colon I \pcoprod I \to A \pcoprod B$ is phase epic.

\end{itemize}
\end{definition}

Let us say that phased coproducts or coproducts in a category are \emph{monic} whenever all coprojections are monic.
In this case we write $[-] \colon \plusI{\catC} \to \catC$ for the functor sending $\obb{A} \mapsto A$ and $f \colon \obb{A} \to \obb{B}$ to the unique $[f] \colon A \to B$ with $f \circ \pcoproj_A = \pcoproj_B \circ [f]$.

\begin{theorem} \label{thm:localToGlobal}
Let $\catC$ be a category with finite monic phased coproducts with transitive phases and a phase generator $\Gen$. Then $\plusI{\catC}$ has monic finite coproducts. Moreover, it has a choice of trivial isomorphisms 
\[
\TrI_\obb{A} := \{ U \colon \obb{A} \to \obb{A} \mid U \text{ is a phase}\}
\]
whose congruence $\sim$ induces an equivalence of categories
\[
\catC \simeq \quot{\plusI{\catC}}{\sim}
\]
\end{theorem}

\begin{proof}
Note that any initial object $0$ in $\ctb$ forms an initial object $\obb{0} = 0 \pcoprod \Gen$ in $\plusI{\ctb}$. Indeed any morphism $f \colon \obb{0} \to \obb{A}$ preserves the $\pcoproj_\Gen$, but by Proposition~\ref{prop:phase_init_unit} $\pcoproj_\Gen \colon \Gen \to \obb{0}$ is an isomorphism, making $f$ unique. 

Now for any pair of objects $\obb{A} = A \pcoprod \Gen$, $\obb{B} = B \pcoprod \Gen$ in $\plusI{\ctb}$ we claim that any phased coproduct $A \pcoprod B$ and object 
\[
\obb{A} + \obb{B} := (A \pcoprod B) \pcoprod \Gen
\]
and morphisms $\pcoproj_{A,\Gen} \colon \obb{A} \to \obb{A} + \obb{B}$ and $\pcoproj_{B,\Gen} \colon \obb{B} \to \obb{A} + \obb{B}$ with $[\pcoproj_{A,\Gen}] = \pcoproj_A$ and $[\pcoproj_{B,\Gen}] = \pcoproj_B$ forms their coproduct in $\plusI{\ctb}$. These morphisms are special kinds of coprojections by associativity (Proposition~\ref{prop:assoc}) and so in particular are monic. We need to show that for all morphisms $f, g$ belonging to $\plusI{\ctb}$ that in $\ctb$ there is a unique $h$ making the following commute:
\[
\begin{tikzcd}
& (A \pcoprod B) \pcoprod \Gen \arrow["h", d, dotted] & \\
 A \pcoprod \Gen \arrow[swap,"f", r] \arrow[ur, "\pcoproj_{A,\Gen}"]  & C \pcoprod \Gen & \Gen \pcoprod B \arrow[ul, swap,"\pcoproj_{B,\Gen}"] \arrow["g", l] 
\end{tikzcd}
\]
We start with the existence property. By Proposition~\ref{prop:assoc} $(A \pcoprod B) \pcoprod \Gen$ also forms a phased coproduct of $A \pcoprod \Gen$ and $B$ via $\pcoproj_{A,\Gen}$ and $\pcoproj_B = \pcoproj_{A, \Gen} \circ \pcoproj_B$. So there exists $k \colon (A \pcoprod B) \pcoprod \Gen \to C \pcoprod I$ with $k \circ \pcoproj_{A,\Gen} = f$ and $k \circ \pcoproj_B = g \circ \pcoproj_B$. Then
\[
k \circ \pcoproj_\Gen = k \circ \pcoproj_{A,\Gen} \circ \pcoproj_\Gen = f \circ \pcoproj_\Gen = g \circ \pcoproj_\Gen
\]
also, and so $k \circ \pcoproj_{B,\Gen} = g \circ U$ for some phase $U$ on $B \pcoprod \Gen$. By transitivity there then is a phase $V$ with respect to $\pcoproj_{A, \Gen}$, $\pcoproj_B$ for which $\pcoproj_{B,\Gen} \circ U = V \circ \pcoproj_{B,\Gen}$. Then $h= k \circ V^{-1}$ is easily seen to have the desired properties.

We next verify uniqueness. Suppose that there exists $f, g$ with $f \circ \pcoproj_{A,\Gen} = g \circ \pcoproj_{A,\Gen}$ and $f \circ \pcoproj_{B,\Gen} = g \circ \pcoproj_{B,\Gen}$. Consider morphisms $h, j$ as in the diagram 
\[
\begin{tikzcd}
\Gen \pcoprod \Gen \dar[swap]{\triangledown} \arrow[loop left, "V", distance=2em] \rar{j} & (A \pcoprod \Gen) \pcoprod (B \pcoprod \Gen) \arrow[loop right, "U", distance=4em] \dar{h} & \\ 
\Gen \rar[swap]{\pcoproj_\Gen} & (A \pcoprod B) \pcoprod \Gen \arrow[r, shift left=2.5, "f"] \arrow[r, shift right=2.5, "g", swap] & C \pcoprod \Gen
\end{tikzcd}
\]
with
\begin{align*}
&h \circ \pcoproj_{A \pcoprod \Gen} = \pcoproj_{A,\Gen}
&j \circ \pcoproj_1 = \pcoproj_{A,\Gen} \circ \pcoproj_\Gen
\\
&h \circ \pcoproj_{B \pcoprod \Gen} = \pcoproj_{B,\Gen}
&j \circ \pcoproj_2 = \pcoproj_{B,\Gen} \circ \pcoproj_\Gen
\end{align*}
Then $g \circ h = f \circ h \circ U$ for some phase $U$ on $(A \pcoprod \Gen) \pcoprod (\Gen \pcoprod B)$. Now $h$ is split epic, since $h \circ k$ is a phase whenever $k$ is a morphism in the opposite direction defined via any of the obvious inclusions of $A, B$ and $I$ into each object. Hence it suffices to prove that $U=\id{}$.

Since $j$ is diagonal we have $U \circ j = j \circ V$ for some phase $V$ as above. We first show that $V = \id{\Gen \pcoprod \Gen}$. Composing with coprojections shows that $h \circ j  = \coproj_\Gen \circ \triangledown$ for some $\triangledown$ with $\triangledown \circ \pcoproj_1 = \triangledown \circ \pcoproj_2 = \id{\Gen}$. Then we have
\begin{align*}
\pcoproj^{C \pcoprod \Gen}_\Gen \circ \triangledown 
&= 
g \circ \pcoproj_\Gen \circ \triangledown
\\
&=
g \circ h \circ j
\\ 
&=
f \circ h \circ U \circ j
\\
&=
f \circ h \circ j \circ V
=
\pcoproj^{C \pcoprod \Gen}_\Gen \circ \triangledown \circ V
\end{align*}
and so $\triangledown  = \triangledown \circ V$. Then since $I$ is a phase generator $V = \id{\Gen \pcoprod \Gen}$, so that $U \circ j = j$. Now again by associativity of phased coproducts $j$ is a coprojection and so is monic, and then since it is diagonal and $I$ is a phase generator we have $U = \id{}$.

For the second statement, note that these $\TrI_\obb{A}$ are a valid choice of trivial isomorphisms, satisfying \eqref{eq:triv-isom} since all morphisms in $\plusI{\ctb}$ are diagonal in $\ctb$. Moreover we indeed have $[f]_\tc = [g]_\tc$ whenever $[f] = [g]$ for the functor $[-] \colon \plusI{\ctb} \to \ctb$. Hence $[-]$ restricts along $[-]_\tc$ to an equivalence $\quot{\plusI{\ctb}}{\sim} \simeq \ctb$.
\end{proof}

\section{Monoidal Categories} \label{sec:monoidal}

Several of our motivating examples of phased coproducts belong to categories with a compatible monoidal structure, and we will see that in this setting the $\plusI{-}$ construction is a natural one.

First, say that a functor $F \colon \catC \to \catC'$ \emph{strongly preserves} phased coproducts if for every phased coproduct $A \pcoprod B$ with coprojections $\pcoproj_A, \pcoproj_B$ in $\catC$, $F(A \pcoprod B)$ is a phased coproduct with coprojections $F(\pcoproj_A)$, $F(\pcoproj_B)$ and moreover has that every phase is of the form $F(U)$ for some phase $U$ of $A \pcoprod B$.

\begin{definition} \label{def:distirb}
We say that phased coproducts in a monoidal category are \emph{distributive} when they are strongly preserved by the functors $A \otimes (-)$ and $(-) \otimes A$, for all objects $A$.
\end{definition}

Thanks to Lemma~\ref{lem:isoms} the requirement on $A \otimes (-)$ is equivalent to requiring that some (and hence any) morphism 
\begin{equation} \label{eq:distrib-isomorphism}
 \begin{tikzcd}
A \otimes B \pcoprod A \otimes C \rar{f} & A \otimes (B \pcoprod C)
\end{tikzcd}
\qquad
\text{with}
\qquad
\begin{tabular}{cc}
$f \circ \pcoproj_{A \otimes B} = \id{A} \otimes \pcoproj_B$
\\ 
$f \circ \pcoproj_{A \otimes C} = \id{A} \otimes \pcoproj_C$
\end{tabular}
\end{equation}
 is an isomorphism, and moreover has that every phase on its domain is of the form $f^{-1} \circ (\id{A} \otimes U) \circ f$ for some phase $U$ of $B \pcoprod C$.

In the case of actual coproducts, this specialises to the usual notion of distributivity, with the phase condition redundant.  

\begin{remark}
Our definition of distributivity, requiring from strong preservation that \emph{every} phase of $A \otimes B \pcoprod A \otimes C$ arises from one of $B \pcoprod C$, may indeed appear rather strong. However we will find it to hold in very general quotient categories, and in Section \ref{sec:compact_cats} to be automatic in any compact category.  
\end{remark}

Now the trivial isomorphisms in our main examples may be defined naturally using their monoidal structure as follows. Recall that in any monoidal category $\catC$ the morphisms $s \in \catC(I,I)$, called \emph{scalars}, form a commutative monoid. They come with a left action $f \mapsto s \cdot f$  and right action $f \mapsto f \cdot s$ on each homset $\catC(A,B)$. Let us call a scalar $s$ \emph{central} when these coincide, i.e.~we have $s \cdot f = f \cdot s$, for all morphisms $f$ in $\catC$. In a braided or symmetric monoidal category every scalar is central.

\begin{definition}
 By a choice of \deff{global phases} in a monoidal category $\cta$ we mean a collection $\mathbb{P}$ of invertible, central scalars closed under composition and inverses. 
\end{definition}

Any such global phase group $\mathbb{P}$ determines a choice of trivial isomorphisms on $\catC$ by setting $\TrI_A := \{ p \cdot \id{A} \mid p \in \mathbb{P}\}$. Then $\TrI_I = \mathbb{P}$, the induced congruence $\sim$ is given by  
\begin{equation} \label{eq:quotient_rule_general}
f \sim g \text{ if } f = u \cdot g \text{ for some $u \in \mathbb{P}$}
\end{equation}
and we write $\cta_\quotP := \quot{\cta}{\sim}$.

\begin{lemma} \label{lem:globaltolocalmonoidal}
Let $\catC$ be a monoidal category with distributive finite coproducts and a choice of global phases $\mathbb{P}$. Then $\catC_\quotP$ is a monoidal category with distributive finite phased coproducts with transitive phases.
\end{lemma}
\begin{proof}
Since the $p \in \mathbb{P}$ are central we have $f \sim h$, $g \sim k$ $\implies$ $f \otimes g \sim h \otimes k$. Hence $\otimes$ restricts from $\catC$ to $\catC_\quotP$, making the latter category monoidal. By Lemma~\ref{lem:generalrecipeforphcoprod} coproducts in $\catC$ become phased coproducts in $\catC_\quotP$. It is easy to see that transitivity holds and that distributivity is inherited from $\catC$.
\end{proof}

\begin{examples} \label{example:global-phases}
$\VecSp_{}$ and $\Hilb$ are monoidal categories with distributive finite coproducts, and our earlier choice of trivial isomorphisms correspond to the global phase group $\mathbb{P} = \{ e^{i \theta} \mid \theta \in [0,2\pi)\}$ in both cases. Similarly $\FVecSp_{k}$ is monoidal and its choice of trivial isomorphisms comes from the global phase group $\mathbb{P} = \{ \lambda \in k \mid  \lambda \neq 0\}$.
\end{examples}

We now wish to give a converse to this result, showing that $\plusI{\catC}$ is a monoidal category with a canonical choice of global phases. When $\catC$ is monoidal we'll always take as chosen object $\Gen$ its monoidal unit. To prove monoidality of $\plusI{\catC}$  we will use the following general result from~\cite[Prop.~2.6, Lemma~2.7]{kock2008elementary}. 

\begin{lemma}\label{lem:mon_cat_helpful}
A monoidal structure on a category $\catC$ is equivalent to specifying:
\begin{itemize}
\item  a bifunctor $\otimes$ and natural isomorphism $\alpha$ satisfying the pentagon equation;
\item 
an object $I$ such that every morphism $A \otimes I \to B \otimes I$ and $I \otimes A \to I \otimes B$ is of the form $g \otimes \id{I}$, $\id{I} \otimes g$ respectively, for some unique $g \colon A \to B$;
\item 
an isomorphism $\beta \colon I \otimes I \isomto I$.
\end{itemize}
\end{lemma}

We will also repeatedly use the following elementary observation.

\begin{lemma} \label{lem:useful}
Suppose that we have morphisms
\[
\begin{tikzcd}
A \pcoprod B \rar{f} & E & C \pcoprod D \lar[swap]{g}
\end{tikzcd}
\]
with $f \circ \pcoproj_A = g \circ \pcoproj_C \circ h$ and $f \circ \pcoproj_B = g \circ \pcoproj_D \circ k$ for some $h, k$. Then $f = g \circ l$ for some diagonal morphism $l$.
\end{lemma}
\begin{proof}
Let $m \colon A \pcoprod B \to C \pcoprod D$ have $m \circ \coproj_A = \coproj_C \circ h$ and $m \circ \coproj_B = \coproj_D \circ k$. Then $f = g \circ m \circ U$ for some phase $U$, giving $l = m \circ U$ as the desired morphism. 
\end{proof}

In our proofs we will make use of the \emph{graphical calculus} for monoidal categories (see~\cite{selinger2011survey} for details). Morphisms $f \colon A \to B$ are drawn as 
$\begin{pic}
  \node[box] (f) at (0,0) {$f$};
  \draw (f.south) to +(0,-.1)node[below] {$A$};
  \draw (f.north) to +(0,.1) node[above] {$B$};
\end{pic}$, with
\[
\scalebox{0.8}{\begin{tikzpicture}
	\begin{pgfonlayer}{nodelayer}
		\node [style=box] (0) at (0, -0) {$\id{A}$};
		\node [style=none] (1) at (0, 1.5) {};
		\node [style=none] (2) at (0, -1.5) {};
		\node [style=none] (3) at (2, -0) {$=$};
		\node [style=none] (4) at (4, 1.5) {};
		\node [style=label] (5) at (0, -2) {$A$};
		\node [style=label] (6) at (4, -2) {$A$};
		\node [style=label] (7) at (4, 2) {$A$};
		\node [style=label] (8) at (0, 2) {$A$};
		\node [style=none] (9) at (4, -1.5) {};
	\end{pgfonlayer}
	\begin{pgfonlayer}{edgelayer}
		\draw (0) to (1.center);
		\draw (0) to (2.center);
		\draw (4.center) to (9.center);
	\end{pgfonlayer}
\end{tikzpicture}}

\qquad \qquad
\scalebox{0.8}{\begin{tikzpicture}
	\begin{pgfonlayer}{nodelayer}
		\node [style=box] (0) at (0, -0) {$g \circ f$};
		\node [style=none] (1) at (0, 1.5) {};
		\node [style=none] (2) at (0, -1.5) {};
		\node [style=none] (3) at (2, -0) {$=$};
		\node [style=none] (4) at (4, 1.5) {};
		\node [style=box] (5) at (4, -0.75) {$f$};
		\node [style=box] (6) at (4, 0.75) {$g$};
		\node [style=label] (7) at (0, -2) {$A$};
		\node [style=label] (8) at (4, -2) {$A$};
		\node [style=label] (9) at (4, 2) {$C$};
		\node [style=label] (10) at (0, 2) {$C$};
		\node [style=none] (11) at (4, -1.5) {};
	\end{pgfonlayer}
	\begin{pgfonlayer}{edgelayer}
		\draw (0) to (1.center);
		\draw (5) to (6);
		\draw (6) to (4.center);
		\draw (0) to (2.center);
		\draw (11.center) to (5);
	\end{pgfonlayer}
\end{tikzpicture}}

\qquad \qquad
\scalebox{0.8}{\begin{tikzpicture}
	\begin{pgfonlayer}{nodelayer}
		\node [style=box] (0) at (0, -0) {$f \otimes g$};
		\node [style=none] (1) at (0, 1.5) {};
		\node [style=none] (2) at (0, -1.5) {};
		\node [style=none] (3) at (2, -0) {$=$};
		\node [style=none] (4) at (4, 1.5) {};
		\node [style=box] (5) at (4, -0) {$f$};
		\node [style=label] (6) at (0, -2) {$A \otimes C$};
		\node [style=label] (7) at (4, -2) {$A$};
		\node [style=label] (8) at (4, 2) {$B$};
		\node [style=label] (9) at (0, 2) {$B \otimes D$};
		\node [style=none] (10) at (4, -1.5) {};
		\node [style=box] (11) at (6, -0) {$g$};
		\node [style=label] (12) at (6, 2) {$D$};
		\node [style=none] (13) at (6, -1.5) {};
		\node [style=label] (14) at (6, -2) {$C$};
		\node [style=none] (15) at (6, 1.5) {};
	\end{pgfonlayer}
	\begin{pgfonlayer}{edgelayer}
		\draw (0) to (1.center);
		\draw (0) to (2.center);
		\draw (5) to (4.center);
		\draw (5) to (10.center);
		\draw (11) to (13.center);
		\draw (11) to (15.center);
	\end{pgfonlayer}
\end{tikzpicture}}

\]
The (identity on) the monoidal unit object $I$ is the empty picture, so that morphisms $I \to A$ and $A \to I$ are boxes with no inputs or outputs, respectively.
In a braided or symmetric setting, the swap map $\sigma$ is depicted $\begin{aligned}\begin{pic}[scale=.25]
  \draw (0,-.5) to[out=80,in=-100] (1,.5);
  \draw (1,-.5) to[out=100,in=-80] (0,.5);
\end{pic}\end{aligned}$.

\begin{theorem} \label{thm:constr_is_monoidal}
Let $\catC$ be a monoidal category with distributive monic finite phased coproducts. 
Then $\plusI{\catC}$ is a monoidal category, and $[-] \colon \plusI{\catC} \to \catC$ is a strict monoidal functor. 
\end{theorem}

\begin{proof}
We define a monoidal product $\tens$ on $\plusI{\catC}$ as follows. For each pair of objects $\obb{A}, \obb{B}$ choose some object $\obb{A} \tens \obb{B} = A \otimes B \pcoprod I$ and $c_{\obb{A},\obb{B}} \colon \obb{A} \tens \obb{B} \to \obb{A} \otimes \obb{B}$ satisfying
\begin{equation} \label{eq:corner}
c_{\obb{A},\obb{B}} \circ \pcoproj_{A \otimes B} = \pcoproj_A \otimes \pcoproj_B 
\qquad 
c_{\obb{A},\obb{B}} \circ \pcoproj_I = (\pcoproj_I \otimes \pcoproj_I) \circ \rho_{I}^{-1}
\end{equation}
which we depict as 
\[
\scalebox{0.8}{\begin{tikzpicture}
	\begin{pgfonlayer}{nodelayer}
		\node [style=none] (0) at (-1, 1) {};
		\node [style=none] (1) at (0, -0) {};
		\node [style=label] (2) at (1, 1.5) {$\obb{B}$};
		\node [style=corner2] (3) at (0, -0) {};
		\node [style=none] (4) at (1, 1) {};
		\node [style=label] (5) at (-1, 1.5) {$\obb{A}$};
		\node [style=label] (6) at (0, -1.5) {$\obb{A} \tens \obb{B}$};
		\node [style=none] (7) at (-1, 1) {};
		\node [style=none] (8) at (1, 1) {};
		\node [style=none] (9) at (0, -0) {};
		\node [style=none] (10) at (0, -1) {};
	\end{pgfonlayer}
	\begin{pgfonlayer}{edgelayer}
		\draw [style=none, bend left=45, looseness=1.25] (1.center) to (7.center);
		\draw [style=none, bend right=45, looseness=1.25] (9.center) to (8.center);
		\draw [style=none] (3) to (10.center);
	\end{pgfonlayer}
\end{tikzpicture}}

\]
Using distributivity, associativity (Proposition~\ref{prop:assoc}), and $\rho_I$, we have isomorphisms
\begin{align*}
\obb{A} \otimes \obb{B} 
&\simeq
(A \otimes B \pcoprod A \otimes I) \pcoprod (I \otimes B \pcoprod I \otimes I) \\ 
&\simeq 
(\obb{A} \tens \obb{B})\pcoprod (A \otimes I \pcoprod I \otimes B)
\end{align*}
making any such morphism $c_{\obb{A},\obb{B}}$ a coprojection, and hence monic. Then for morphisms $f \colon \obb{A} \to \obb{C}$ and $g \colon \obb{B} \to \obb{D}$ in $\plusI{\catC}$ we define $f \tens g$ to be the unique morphism in $\catC$ such that 
\[
\scalebox{0.8}{\begin{tikzpicture}
	\begin{pgfonlayer}{nodelayer}
		\node [style=none] (0) at (13, 1.75) {};
		\node [style=label] (1) at (7.25, 2.25) {$\obb{C}$};
		\node [style=box] (2) at (8, -1) {$f \tens g$};
		\node [style=label] (3) at (8, -2.75) {$\obb{A} \tens \obb{B}$};
		\node [style=none] (4) at (8.75, 1) {};
		\node [style=none] (5) at (11.5, -0.25) {};
		\node [style=none] (6) at (8, 0.25) {};
		\node [style=none] (7) at (13, -0.25) {};
		\node [style=label] (8) at (8.75, 2.25) {$\obb{D}$};
		\node [style=none] (9) at (7.25, 1.75) {};
		\node [style=none] (10) at (8.75, 1.75) {};
		\node [style=none] (11) at (8, -2.25) {};
		\node [style=box] (12) at (11.5, 0.75) {$f$};
		\node [style=box] (13) at (13, 0.75) {$g$};
		\node [style=none] (14) at (10, -1.5) {$=$};
		\node [style=label] (15) at (13, 2.25) {$\obb{D}$};
		\node [style=label] (16) at (11.5, 2.25) {$\obb{C}$};
		\node [style=none] (17) at (11.5, 1.75) {};
		\node [style=corner2] (18) at (12.25, -1) {};
		\node [style=none] (19) at (12.25, -2.25) {};
		\node [style=none] (20) at (13, -0.25) {};
		\node [style=none] (21) at (11.5, -0.25) {};
		\node [style=none] (22) at (12.25, -1) {};
		\node [style=none] (23) at (12.25, -1) {};
		\node [style=none] (24) at (8.75, 1) {};
		\node [style=none] (25) at (8, 0.25) {};
		\node [style=corner2] (26) at (8, 0.25) {};
		\node [style=none] (27) at (8, 0.25) {};
		\node [style=none] (28) at (7.25, 1) {};
		\node [style=none] (29) at (8.75, 1.25) {};
		\node [style=label] (30) at (12.25, -2.75) {$\obb{A} \tens \obb{B}$};
		\node [style=label] (31) at (13.5, -0.25) {$\obb{B}$};
		\node [style=label] (32) at (11, -0.25) {$\obb{A}$};
	\end{pgfonlayer}
	\begin{pgfonlayer}{edgelayer}
		\draw (10.center) to (4.center);
		\draw [style=none] (11.center) to (6.center);
		\draw (5.center) to (17.center);
		\draw (7.center) to (0.center);
		\draw [style=none, bend left=45, looseness=1.25] (23.center) to (21.center);
		\draw [style=none, bend right=45, looseness=1.25] (22.center) to (20.center);
		\draw [style=none] (18) to (19.center);
		\draw [style=none, bend left=45, looseness=1.25] (27.center) to (28.center);
		\draw [style=none, bend right=45, looseness=1.25] (25.center) to (24.center);
		\draw [style=none] (9.center) to (28.center);
	\end{pgfonlayer}
\end{tikzpicture}}

\]
Indeed such a map exists and belongs to $\plusI{\catC}$ by Lemma~\ref{lem:useful} since we have
\begin{align*}
(f \otimes g) \circ c_{\obb{A},\obb{B}} \circ \pcoproj_{A \otimes B} &= c_{\obb{C}, \obb{D}} \circ \pcoproj_{C \otimes D} \circ ([f] \otimes [g]) \\ 
(f \otimes g) \circ c_{\obb{A},\obb{B}} \circ \pcoproj_{I \otimes I} &= c_{\obb{C},\obb{D}} \circ \pcoproj_I
\end{align*}
Uniqueness follows from monicity of $c_{\obb{C},\obb{D}}$ and ensures that $\tens$ is functorial. We define $\aalpha_{\obb{A},\obb{B},\obb{C}} \colon (\obb{A} \tens \obb{B}) \tens \obb{C}) \to \obb{A} \tens (\obb{B} \tens \obb{C})$ to be the unique morphism such that 
\begin{equation} \label{eq:defaaalpha}
\scalebox{0.8}{\begin{tikzpicture}
	\begin{pgfonlayer}{nodelayer}
		\node [style=none] (0) at (0, 0) {$=$};
		\node [style=none] (1) at (-3, 0.25) {};
		\node [style=none] (2) at (-5, 0.25) {};
		\node [style=none] (3) at (-3, 0.25) {};
		\node [style=label] (4) at (-5, 2) {$\obb{A}$};
		\node [style=none] (5) at (-5, 1.5) {};
		\node [style=corner2] (6) at (-4, -0.75) {};
		\node [style=none] (7) at (-4, -2.25) {};
		\node [style=none] (8) at (-3, 0.25) {};
		\node [style=none] (9) at (-5, 0.25) {};
		\node [style=none] (10) at (-4, -0.75) {};
		\node [style=none] (11) at (-4, -0.75) {};
		\node [style=label] (12) at (-2, -1.25) {$\obb{A} \dtens (\obb{B} \dtens \obb{C})$};
		\node [style=label] (13) at (-2, 0.25) {$\obb{B} \dtens \obb{C}$};
		\node [style=none] (14) at (-3.75, 1.5) {};
		\node [style=none] (15) at (-3, 0.75) {};
		\node [style=label] (16) at (-2.25, 2) {$\obb{C}$};
		\node [style=corner2] (17) at (-3, 0.75) {};
		\node [style=none] (18) at (-2.25, 1.5) {};
		\node [style=label] (19) at (-3.75, 2) {$\obb{B}$};
		\node [style=none] (20) at (-3.75, 1.5) {};
		\node [style=none] (21) at (-2.25, 1.5) {};
		\node [style=none] (22) at (-3, 0.75) {};
		\node [style=none] (23) at (-3, 0.25) {};
		\node [style=none] (24) at (2.5, 1) {};
		\node [style=none] (25) at (3.25, -0.75) {};
		\node [style=none] (26) at (4, 1) {};
		\node [style=none] (27) at (5.25, 1.5) {};
		\node [style=none] (28) at (4.25, -3.25) {};
		\node [style=none] (29) at (5.25, -0.75) {};
		\node [style=none] (30) at (3.25, 0) {};
		\node [style=none] (31) at (4.25, -1.75) {};
		\node [style=none] (32) at (5.25, -0.75) {};
		\node [style=none] (33) at (3.25, 0) {};
		\node [style=none] (34) at (3.25, 0.25) {};
		\node [style=none] (35) at (4, 1) {};
		\node [style=corner2] (36) at (3.25, 0.25) {};
		\node [style=label] (37) at (2, -0.25) {$\obb{A} \dtens \obb{B}$};
		\node [style=corner2] (38) at (4.25, -1.75) {};
		\node [style=label] (39) at (5.25, 2) {$\obb{C}$};
		\node [style=label] (40) at (2.5, 2) {$\obb{A}$};
		\node [style=none] (41) at (3.25, -0.75) {};
		\node [style=label] (42) at (4.25, -3.75) {$(\obb{A} \dtens \obb{B}) \dtens \obb{C}$};
		\node [style=none] (43) at (3.25, 0.25) {};
		\node [style=none] (44) at (4.25, -1.75) {};
		\node [style=label] (45) at (4, 2) {$\obb{B}$};
		\node [style=none] (46) at (2.5, 1) {};
		\node [style=box] (47) at (-4, -2.25) {$\aalpha_{\obb{A},\obb{B},\obb{C}}$};
		\node [style=none] (48) at (-4, -3.25) {};
		\node [style=label] (49) at (-4, -3.75) {$(\obb{A} \dtens \obb{B}) \dtens \obb{C}$};
		\node [style=none] (50) at (2.5, 1.5) {};
		\node [style=none] (51) at (4, 1.5) {};
	\end{pgfonlayer}
	\begin{pgfonlayer}{edgelayer}
		\draw (2.center) to (5.center);
		\draw (3.center) to (1.center);
		\draw [style=none, bend left=45, looseness=1.25] (11.center) to (9.center);
		\draw [style=none, bend right=45, looseness=1.25] (10.center) to (8.center);
		\draw [style=none] (6) to (7.center);
		\draw [style=none, bend left=45] (15.center) to (20.center);
		\draw [style=none, bend right=45] (22.center) to (21.center);
		\draw [style=none] (17) to (23.center);
		\draw (32.center) to (27.center);
		\draw (41.center) to (30.center);
		\draw [style=none, bend right=45, looseness=1.25] (44.center) to (29.center);
		\draw [style=none, bend left=45, looseness=1.25] (31.center) to (25.center);
		\draw [style=none] (38) to (28.center);
		\draw [style=none, bend right=45, looseness=1.25] (34.center) to (35.center);
		\draw [style=none, bend left=45, looseness=1.25] (43.center) to (46.center);
		\draw [style=none] (36) to (33.center);
		\draw [style=none] (48.center) to (7.center);
		\draw [style=none] (51.center) to (26.center);
		\draw [style=none] (50.center) to (24.center);
	\end{pgfonlayer}
\end{tikzpicture}}

\end{equation}
Existence again follows from Lemma~\ref{lem:useful}. For uniqueness, distributivity tell us us that for any $\id{A} \otimes c_{\obb{B}, \obb{C}}$ is again a coprojection since $c_{\obb{B}, \obb{C}}$ is, and hence is monic. By symmetry there is some $\aalpha'_{\obb{A}, \obb{B}, \obb{C}}$ satisfying the horizontally reflected version of~\eqref{eq:defaaalpha}, and then thanks to uniqueness this is an inverse to $\aalpha_{\obb{A}, \obb{B}, \obb{C}}$. 

Again using monicity of the $\id{\obb{A}} \otimes c_{\obb{B},\obb{C}}$ we verify that $\aalpha$ is natural:
\[
\scalebox{0.8}{\input{./figures/sup5il.tikz}}

\]
and that it satisfies the pentagon equation:
\begingroup
\allowdisplaybreaks
\begin{align*}
\scalebox{0.8}{\input{./figures/sup6il-1.tikz}}
 \\ 
\scalebox{0.8}{\begin{tikzpicture}
	\begin{pgfonlayer}{nodelayer}
		\node [style=none] (175) at (2.75, 0.25) {};
		\node [style=none] (176) at (3.25, 3.25) {};
		\node [style=none] (177) at (3, -3.25) {};
		\node [style=none] (178) at (4, 2.5) {};
		\node [style=none] (179) at (2.5, 0) {};
		\node [style=none] (180) at (3, -3.25) {};
		\node [style=none] (181) at (0, 0) {$=$};
		\node [style=none] (182) at (2.75, 4) {};
		\node [style=none] (183) at (3, 0.5) {};
		\node [style=corner2] (184) at (3.25, 3.25) {};
		\node [style=none] (185) at (2.5, 1) {};
		\node [style=none] (186) at (3, 0.5) {};
		\node [style=none] (187) at (2, 0.5) {};
		\node [style=none] (188) at (3.75, 3.75) {};
		\node [style=none] (189) at (2.75, 3.75) {};
		\node [style=none] (190) at (3.25, 3.25) {};
		\node [style=none] (191) at (4.75, 4) {};
		\node [style=box] (192) at (3, -2.25) {$\aalpha_{\obb{A},\obb{B},\obb{C}} \dtens \id{\obb{D}}$};
		\node [style=none] (193) at (3.75, 4) {};
		\node [style=none] (194) at (4.75, 3.25) {};
		\node [style=corner2] (195) at (4, 2.5) {};
		\node [style=none] (196) at (3.25, 3.25) {};
		\node [style=none] (197) at (4.75, 3.25) {};
		\node [style=none] (198) at (3, 0.5) {};
		\node [style=none] (199) at (3.25, 3.25) {};
		\node [style=none] (200) at (3.25, 3.25) {};
		\node [style=corner2] (201) at (3, 0.5) {};
		\node [style=none] (202) at (4, 1.5) {};
		\node [style=none] (203) at (4, 2.5) {};
		\node [style=none] (204) at (2, 1.5) {};
		\node [style=none] (205) at (2, 4.25) {};
		\node [style=box] (206) at (3, -0.75) {$\aalpha_{\obb{A}, \obb{B} \dtens \obb{C}, \obb{D}}$};
		\node [style=none] (207) at (9.75, 2.75) {};
		\node [style=corner2] (208) at (9.75, 2.75) {};
		\node [style=none] (209) at (9, 3.5) {};
		\node [style=corner2] (210) at (8.75, 0.5) {};
		\node [style=none] (211) at (9, 3.5) {};
		\node [style=none] (212) at (8.25, 0) {};
		\node [style=none] (213) at (10.5, 3.5) {};
		\node [style=none] (214) at (7.75, 1.5) {};
		\node [style=corner2] (215) at (10.5, 3.5) {};
		\node [style=none] (216) at (8.25, 1) {};
		\node [style=none] (217) at (9.75, 1.5) {};
		\node [style=none] (218) at (7.75, 4.25) {};
		\node [style=none] (219) at (10, 4) {};
		\node [style=none] (220) at (10.5, 3.5) {};
		\node [style=none] (221) at (10.5, 3.5) {};
		\node [style=none] (222) at (10.5, 3.5) {};
		\node [style=none] (223) at (11, 4) {};
		\node [style=none] (224) at (8.75, 0.5) {};
		\node [style=none] (225) at (9, 4.25) {};
		\node [style=none] (226) at (10.5, 3.5) {};
		\node [style=none] (227) at (8.75, -3.25) {};
		\node [style=none] (228) at (7.75, 0.5) {};
		\node [style=none] (229) at (8.75, 0.5) {};
		\node [style=none] (230) at (8.75, 0.5) {};
		\node [style=none] (231) at (10, 4.25) {};
		\node [style=none] (232) at (11, 4.25) {};
		\node [style=none] (233) at (9.75, 2.75) {};
		\node [style=none] (234) at (8.75, -3.25) {};
		\node [style=none] (235) at (8.5, 0.25) {};
		\node [style=none] (236) at (5.75, 0) {$=$};
		\node [style=box] (237) at (9.75, 1.5) {$\aalpha_{\obb{B},\obb{C},\obb{D}}$};
		\node [style=none] (238) at (15.75, -3.25) {};
		\node [style=none] (239) at (15.75, 1.75) {};
		\node [style=none] (240) at (15.75, -3.25) {};
		\node [style=corner2] (241) at (15.75, 1.75) {};
		\node [style=corner2] (242) at (16.75, 2.75) {};
		\node [style=none] (243) at (15.25, 1.25) {};
		\node [style=none] (244) at (16.75, 2.75) {};
		\node [style=none] (245) at (17.5, 3.5) {};
		\node [style=none] (246) at (15.75, 1.75) {};
		\node [style=none] (247) at (16, 3.5) {};
		\node [style=none] (248) at (12.5, 0) {$=$};
		\node [style=none] (249) at (17.5, 3.5) {};
		\node [style=none] (250) at (16, 4.25) {};
		\node [style=none] (251) at (16.75, 2.75) {};
		\node [style=none] (252) at (17.5, 3.5) {};
		\node [style=none] (253) at (16.75, 2.75) {};
		\node [style=none] (254) at (15.5, 1.5) {};
		\node [style=none] (255) at (17, 4.25) {};
		\node [style=none] (256) at (15.75, 1.75) {};
		\node [style=none] (257) at (16, 3.5) {};
		\node [style=none] (258) at (17.5, 3.5) {};
		\node [style=none] (259) at (18, 4.25) {};
		\node [style=none] (260) at (14.75, 1.75) {};
		\node [style=none] (261) at (17.5, 3.5) {};
		\node [style=none] (262) at (17, 4) {};
		\node [style=none] (263) at (14.75, 2.75) {};
		\node [style=none] (264) at (15.25, 2.25) {};
		\node [style=none] (265) at (18, 4) {};
		\node [style=corner2] (266) at (17.5, 3.5) {};
		\node [style=none] (267) at (14.75, 4.25) {};
		\node [style=box] (268) at (15.75, 0.75) {$\id{\obb{A}} \dtens \aalpha_{\obb{B},\obb{C},\obb{D}}$};
		\node [style=box] (282) at (8.75, -2.25) {$\aalpha_{\obb{A},\obb{B},\obb{C}} \dtens \id{\obb{D}}$};
		\node [style=box] (283) at (8.75, -0.75) {$\aalpha_{\obb{A}, \obb{B} \dtens \obb{C}, \obb{D}}$};
		\node [style=box] (284) at (15.75, -2.25) {$\aalpha_{\obb{A},\obb{B},\obb{C}} \dtens \id{\obb{D}}$};
		\node [style=box] (285) at (15.75, -0.75) {$\aalpha_{\obb{A}, \obb{B} \dtens \obb{C}, \obb{D}}$};
		\node [style=label] (286) at (15.75, -3.75) {$((\obb{A} \dtens \obb{B}) \dtens \obb{C}) \dtens \obb{D}$};
		\node [style=label] (287) at (16, 4.75) {$\obb{B}$};
		\node [style=label] (288) at (18, 4.75) {$\obb{D}$};
		\node [style=label] (289) at (17, 4.75) {$\obb{C}$};
		\node [style=label] (290) at (14.75, 4.75) {$\obb{A}$};
		\node [style=label] (291) at (7.75, 4.75) {$\obb{A}$};
		\node [style=label] (292) at (10, 4.75) {$\obb{C}$};
		\node [style=label] (293) at (11, 4.75) {$\obb{D}$};
		\node [style=label] (294) at (9, 4.75) {$\obb{B}$};
		\node [style=label] (295) at (2, 4.75) {$\obb{A}$};
		\node [style=label] (296) at (3.75, 4.75) {$\obb{C}$};
		\node [style=label] (297) at (4.75, 4.75) {$\obb{D}$};
		\node [style=label] (298) at (2.75, 4.75) {$\obb{B}$};
		\node [style=none] (307) at (4, 2.5) {};
		\node [style=none] (308) at (2.75, 4.25) {};
		\node [style=none] (309) at (3.75, 4.25) {};
		\node [style=none] (310) at (4.75, 4.25) {};
	\end{pgfonlayer}
	\begin{pgfonlayer}{edgelayer}
		\draw (197.center) to (191.center);
		\draw [style=none, bend right=45] (178.center) to (194.center);
		\draw [style=none, bend left=45, looseness=1.25] (203.center) to (190.center);
		\draw [style=none, bend right=45] (200.center) to (188.center);
		\draw [style=none, bend left=45] (196.center) to (189.center);
		\draw [style=none] (184) to (176.center);
		\draw [style=none, bend right=45] (186.center) to (202.center);
		\draw [style=none] (177.center) to (198.center);
		\draw [style=none] (188.center) to (193.center);
		\draw [style=none] (189.center) to (182.center);
		\draw [style=none, bend left=45] (183.center) to (204.center);
		\draw [style=none] (204.center) to (205.center);
		\draw (211.center) to (225.center);
		\draw [style=none, bend left=45] (207.center) to (209.center);
		\draw [style=none, bend right=45] (233.center) to (221.center);
		\draw [style=none, bend left=45] (213.center) to (219.center);
		\draw [style=none, bend right=45] (226.center) to (223.center);
		\draw [style=none] (215) to (222.center);
		\draw [style=none, bend right=45] (230.center) to (217.center);
		\draw [style=none] (227.center) to (229.center);
		\draw [style=none] (219.center) to (231.center);
		\draw [style=none] (223.center) to (232.center);
		\draw [style=none, bend left=45] (224.center) to (214.center);
		\draw [style=none] (214.center) to (218.center);
		\draw [style=none] (217.center) to (207.center);
		\draw (257.center) to (250.center);
		\draw [style=none, bend left=45] (253.center) to (247.center);
		\draw [style=none, bend right=45] (251.center) to (249.center);
		\draw [style=none, bend left=45, looseness=1.25] (252.center) to (262.center);
		\draw [style=none, bend right=45] (245.center) to (265.center);
		\draw [style=none] (266) to (261.center);
		\draw [style=none, bend right=45] (239.center) to (244.center);
		\draw [style=none] (240.center) to (246.center);
		\draw [style=none] (262.center) to (255.center);
		\draw [style=none] (265.center) to (259.center);
		\draw [style=none, bend left=45] (256.center) to (263.center);
		\draw [style=none] (263.center) to (267.center);
		\draw [style=none] (244.center) to (253.center);
		\draw [style=none] (202.center) to (307.center);
		\draw [style=none] (310.center) to (191.center);
		\draw [style=none] (309.center) to (193.center);
		\draw [style=none] (308.center) to (182.center);
	\end{pgfonlayer}
\end{tikzpicture}}
 
\end{align*}
\endgroup
For the unit object in $\plusI{\catC}$ choose any $\obb{I} = I \pcoprod I$. Then any morphism $\beta \colon \obb{I} \tens \obb{I} \to \obb{I}$ with $\beta \circ \pcoproj_1 = \pcoproj_1 \circ \rho_I$ and $\beta \circ \pcoproj_2 = \pcoproj_2$ is an isomorphism belonging to $\plusI{\catC}$. 

We now show that in $\plusI{\catC}$ every morphism $f \colon \obb{A} \tens \obb{I} \to \obb{B} \tens \obb{I}$ is of the form $g \tens \id{\obb{I}}$ for a unique $g \colon \obb{A} \to \obb{B}$. Choose any $r_A$, $r_B$ in $\plusI{\catC}$ with $[r_A] = \rho_A$ and $[r_B] = \rho_B$ in $\catC$, setting
\[
\scalebox{0.8}{\begin{tikzpicture}
	\begin{pgfonlayer}{nodelayer}
		\node [style=none] (0) at (0, -0) {$:=$};
		\node [style=none] (1) at (-2.5, 0.5) {};
		\node [style=none] (2) at (-1.5, 0.5) {};
		\node [style=corner3] (3) at (-2, -0) {};
		\node [style=none] (4) at (-1.5, 0.5) {};
		\node [style=none] (5) at (-2.5, 0.5) {};
		\node [style=none] (6) at (-2, -0) {};
		\node [style=none] (7) at (-2, -0) {};
		\node [style=none] (8) at (-1.5, 0.5) {};
		\node [style=none] (9) at (-2, -1.25) {};
		\node [style=none] (10) at (2.25, -1.25) {};
		\node [style=none] (11) at (2.75, 1) {};
		\node [style=none] (12) at (2.75, 1) {};
		\node [style=none] (13) at (2.75, 1) {};
		\node [style=box] (14) at (2.25, -0.5) {$r_{A}$};
		\node [style=none] (15) at (2.25, 0.5) {};
		\node [style=none] (16) at (1.75, 1) {};
		\node [style=none] (17) at (2.25, 0.5) {};
		\node [style=none] (18) at (2.25, 0.5) {};
		\node [style=none] (19) at (1.75, 1) {};
		\node [style=corner2] (20) at (2.25, 0.5) {};
		\node [style=label] (21) at (-2.5, 1.75) {$\obb{A}$};
		\node [style=label] (22) at (-1.5, 1.75) {$\obb{I}$};
		\node [style=label] (23) at (-2, -1.75) {$\obb{A}$};
		\node [style=label] (24) at (1.75, 1.75) {$\obb{A}$};
		\node [style=label] (25) at (2.75, 1.75) {$\obb{I}$};
		\node [style=label] (26) at (2.25, -1.75) {$\obb{A}$};
		\node [style=none] (27) at (-2.5, 1.25) {};
		\node [style=none] (28) at (-1.5, 1.25) {};
		\node [style=none] (29) at (1.75, 1.25) {};
		\node [style=none] (30) at (2.75, 1.25) {};
		\node [style=label] (31) at (11.75, 1.75) {$\obb{B}$};
		\node [style=none] (32) at (12.25, 0.5) {};
		\node [style=none] (33) at (8, -0) {};
		\node [style=none] (34) at (7.5, 1.25) {};
		\node [style=none] (35) at (10, -0) {$:=$};
		\node [style=none] (36) at (12.75, 1) {};
		\node [style=corner2] (37) at (12.25, 0.5) {};
		\node [style=corner3] (38) at (8, -0) {};
		\node [style=box] (39) at (12.25, -0.5) {$r_{B}$};
		\node [style=none] (40) at (7.5, 0.5) {};
		\node [style=label] (41) at (12.75, 1.75) {$\obb{I}$};
		\node [style=none] (42) at (8, -1.25) {};
		\node [style=none] (43) at (8, -0) {};
		\node [style=none] (44) at (12.75, 1) {};
		\node [style=none] (45) at (8.5, 0.5) {};
		\node [style=none] (46) at (8.5, 1.25) {};
		\node [style=none] (47) at (11.75, 1) {};
		\node [style=none] (48) at (12.25, 0.5) {};
		\node [style=label] (49) at (7.5, 1.75) {$\obb{B}$};
		\node [style=none] (50) at (11.75, 1.25) {};
		\node [style=label] (51) at (8, -1.75) {$\obb{B}$};
		\node [style=none] (52) at (8.5, 0.5) {};
		\node [style=none] (53) at (12.25, -1.25) {};
		\node [style=label] (54) at (12.25, -1.75) {$\obb{B}$};
		\node [style=none] (55) at (12.25, 0.5) {};
		\node [style=none] (56) at (7.5, 0.5) {};
		\node [style=label] (57) at (8.5, 1.75) {$\obb{I}$};
		\node [style=none] (58) at (11.75, 1) {};
		\node [style=none] (59) at (12.75, 1.25) {};
		\node [style=none] (60) at (12.75, 1) {};
		\node [style=none] (61) at (8.5, 0.5) {};
	\end{pgfonlayer}
	\begin{pgfonlayer}{edgelayer}
		\draw [style=none, bend left=45, looseness=1.25] (7.center) to (5.center);
		\draw [style=none, bend right=45, looseness=1.25] (6.center) to (4.center);
		\draw [style=none, bend left=45, looseness=1.25] (17.center) to (19.center);
		\draw [style=none, bend right=45, looseness=1.25] (15.center) to (12.center);
		\draw [style=none] (10.center) to (18.center);
		\draw [style=none] (9.center) to (3);
		\draw (2.center) to (28.center);
		\draw (1.center) to (27.center);
		\draw (11.center) to (30.center);
		\draw (16.center) to (29.center);
		\draw [style=none, bend left=45, looseness=1.25] (43.center) to (56.center);
		\draw [style=none, bend right=45, looseness=1.25] (33.center) to (61.center);
		\draw [style=none, bend left=45, looseness=1.25] (55.center) to (58.center);
		\draw [style=none, bend right=45, looseness=1.25] (48.center) to (44.center);
		\draw [style=none] (53.center) to (32.center);
		\draw [style=none] (42.center) to (38);
		\draw (52.center) to (46.center);
		\draw (40.center) to (34.center);
		\draw (36.center) to (59.center);
		\draw (47.center) to (50.center);
	\end{pgfonlayer}
\end{tikzpicture}}

\]
Then the statement is equivalent to requiring that for every diagonal $f \colon \obb{A} \to \obb{B}$ there is a unique diagonal $g \colon \obb{A} \to \obb{B}$ with
\[
\scalebox{0.8}{\begin{tikzpicture}
	\begin{pgfonlayer}{nodelayer}
		\node [style=none] (0) at (0, -2) {$=$};
		\node [style=none] (1) at (-2.5, -0.75) {};
		\node [style=none] (2) at (-1.5, -0.75) {};
		\node [style=corner3] (3) at (-2, -1.25) {};
		\node [style=none] (4) at (-1.5, -0.75) {};
		\node [style=none] (5) at (-2.5, -0.75) {};
		\node [style=none] (6) at (-2, -1.25) {};
		\node [style=none] (7) at (-2, -1.25) {};
		\node [style=none] (8) at (-1.5, -0.75) {};
		\node [style=none] (9) at (-2, -3.25) {};
		\node [style=none] (10) at (2.25, -3.25) {};
		\node [style=none] (11) at (2.75, -1.75) {};
		\node [style=none] (12) at (2.75, -1.75) {};
		\node [style=none] (13) at (2.75, -1.75) {};
		\node [style=none] (14) at (2.25, -2.25) {};
		\node [style=none] (15) at (1.75, -1.75) {};
		\node [style=none] (16) at (2.25, -2.25) {};
		\node [style=none] (17) at (2.25, -2.25) {};
		\node [style=none] (18) at (1.75, -1.75) {};
		\node [style=corner3] (19) at (2.25, -2.25) {};
		\node [style=label] (20) at (-2.5, 0.5) {$\obb{B}$};
		\node [style=label] (21) at (-1.5, 0.5) {$\obb{I}$};
		\node [style=label] (22) at (-2, -3.75) {$\obb{A}$};
		\node [style=label] (23) at (1.75, 0.5) {$\obb{B}$};
		\node [style=label] (24) at (2.75, 0.5) {$\obb{I}$};
		\node [style=label] (25) at (2.25, -3.75) {$\obb{A}$};
		\node [style=none] (26) at (-2.5, -0) {};
		\node [style=none] (27) at (-1.5, -0) {};
		\node [style=none] (28) at (1.75, -1.5) {};
		\node [style=none] (29) at (2.75, -1.5) {};
		\node [style=box] (30) at (-2, -2.25) {$f$};
		\node [style=none] (31) at (1.75, -0) {};
		\node [style=none] (32) at (2.75, -0) {};
		\node [style=box] (33) at (1.75, -1) {$g$};
	\end{pgfonlayer}
	\begin{pgfonlayer}{edgelayer}
		\draw [style=none, bend left=45, looseness=1.25] (7.center) to (5.center);
		\draw [style=none, bend right=45, looseness=1.25] (6.center) to (4.center);
		\draw [style=none, bend left=45, looseness=1.25] (16.center) to (18.center);
		\draw [style=none, bend right=45, looseness=1.25] (14.center) to (12.center);
		\draw [style=none] (10.center) to (17.center);
		\draw [style=none] (9.center) to (3);
		\draw (2.center) to (27.center);
		\draw (1.center) to (26.center);
		\draw (11.center) to (29.center);
		\draw (15.center) to (28.center);
		\draw (15.center) to (31.center);
		\draw (29.center) to (32.center);
	\end{pgfonlayer}
\end{tikzpicture}}

\]
Now let $\tinycounit \colon \obb{I} \to I$ in $\catC$ with $\tinycounit \circ \pcoproj_1 = \id{I} = \tinycounit \circ \pcoproj_2$. Applying coprojections we have 
\[
\scalebox{0.8}{\begin{tikzpicture}
	\begin{pgfonlayer}{nodelayer}
		\node [style=none] (0) at (0.25, -0) {$=$};
		\node [style=none] (1) at (-2.25, 0.75) {};
		\node [style=none] (2) at (-1.25, 0.75) {};
		\node [style=corner3] (3) at (-1.75, 0.25) {};
		\node [style=none] (4) at (-1.25, 0.75) {};
		\node [style=none] (5) at (-2.25, 0.75) {};
		\node [style=none] (6) at (-1.75, 0.25) {};
		\node [style=none] (7) at (-1.75, 0.25) {};
		\node [style=none] (8) at (-1.25, 0.75) {};
		\node [style=none] (9) at (-1.75, -1.5) {};
		\node [style=none] (10) at (-2.25, 1.25) {};
		\node [style=none] (11) at (-1.25, 1.25) {};
		\node [style=corner3] (12) at (-1.25, 1.25) {};
		\node [style=none] (13) at (1.75, 1.25) {};
		\node [style=none] (14) at (1.75, -1.5) {};
		\node [style=corner1] (15) at (-1.75, -0.75) {$U$};
	\end{pgfonlayer}
	\begin{pgfonlayer}{edgelayer}
		\draw [style=none, bend left=45, looseness=1.25] (7.center) to (5.center);
		\draw [style=none, bend right=45, looseness=1.25] (6.center) to (4.center);
		\draw [style=none] (9.center) to (3);
		\draw (2.center) to (11.center);
		\draw (1.center) to (10.center);
		\draw (14.center) to (13.center);
	\end{pgfonlayer}
\end{tikzpicture}}
 
\]
for some phase $U$, which is in particular invertible. This makes $g$ unique. We now show that $g$ exists. 
Applying coprojections again one may see that
\[
\scalebox{0.8}{\begin{tikzpicture}
	\begin{pgfonlayer}{nodelayer}
		\node [style=none] (0) at (1, 0.5) {};
		\node [style=none] (1) at (1.75, -0.25) {};
		\node [style=grey dot] (2) at (1.75, -0.25) {};
		\node [style=none] (3) at (2.5, 0.5) {};
		\node [style=none] (4) at (2.5, 0.5) {};
		\node [style=none] (5) at (2.5, 0.5) {};
		\node [style=none] (6) at (1, 0.5) {};
		\node [style=none] (7) at (1.75, -0.25) {};
		\node [style=none] (8) at (1.75, -2) {};
		\node [style=none] (9) at (1, 0.5) {};
		\node [style=corner1] (10) at (1.75, -1.25) {$V$};
		\node [style=none] (11) at (1, 3) {};
		\node [style=none] (12) at (-1.5, -2) {};
		\node [style=none] (13) at (-1.5, 1) {};
		\node [style=box] (14) at (-1.5, -0.5) {$f$};
		\node [style=none] (15) at (0, -0) {$=$};
		\node [style=none] (16) at (-0.75, 1.75) {};
		\node [style=none] (17) at (-0.75, 1.75) {};
		\node [style=none] (18) at (-2.25, 3) {};
		\node [style=none] (19) at (-0.75, 1.75) {};
		\node [style=none] (20) at (-0.75, 1.75) {};
		\node [style=none] (21) at (-1.5, 1) {};
		\node [style=none] (22) at (-2.25, 1.75) {};
		\node [style=none] (23) at (-0.75, 3) {};
		\node [style=none] (24) at (-1.5, 1) {};
		\node [style=none] (25) at (-2.25, 1.75) {};
		\node [style=none] (26) at (-2.25, 1.75) {};
		\node [style=corner3] (27) at (-1.5, 1) {};
		\node [style=box] (28) at (1, 2.25) {$f$};
		\node [style=none] (29) at (2.5, 3) {};
		\node [style=none] (30) at (7.5, -0) {};
		\node [style=none] (31) at (10.25, -0) {};
		\node [style=none] (32) at (7.5, -0) {};
		\node [style=none] (33) at (10.25, -1.75) {};
		\node [style=none] (34) at (11, 2.25) {};
		\node [style=grey dot] (35) at (10.25, -0) {};
		\node [style=none] (36) at (6, -0) {};
		\node [style=none] (37) at (9.5, 0.75) {};
		\node [style=none] (38) at (6, -0) {};
		\node [style=none] (39) at (6.75, -0.75) {};
		\node [style=none] (40) at (7.5, -0) {};
		\node [style=none] (41) at (11, 0.75) {};
		\node [style=none] (42) at (7.5, 2.5) {};
		\node [style=grey dot] (43) at (6.75, -0.75) {};
		\node [style=none] (44) at (6.75, -0.75) {};
		\node [style=none] (45) at (9.5, 0.75) {};
		\node [style=none] (46) at (9.5, 0.75) {};
		\node [style=none] (47) at (11, 0.75) {};
		\node [style=none] (48) at (9.5, 2.25) {};
		\node [style=none] (49) at (10.25, -0) {};
		\node [style=none] (50) at (8.5, -0) {$=$};
		\node [style=none] (51) at (11, 0.75) {};
		\node [style=corner1] (52) at (10.25, -1) {$W$};
		\node [style=corner1] (53) at (6, 1.5) {$V$};
		\node [style=none] (54) at (11, 0.75) {};
		\node [style=none] (55) at (6.75, -1.75) {};
		\node [style=none] (56) at (6, 2.5) {};
		\node [style=corner1] (57) at (1, 1) {$U$};
		\node [style=corner1] (58) at (6, 0.25) {$U$};
		\node [style=corner1] (59) at (9.5, 1.5) {$U$};
	\end{pgfonlayer}
	\begin{pgfonlayer}{edgelayer}
		\draw [style=none, bend left=45, looseness=1.25] (1.center) to (6.center);
		\draw [style=none, bend right=45, looseness=1.25] (7.center) to (5.center);
		\draw [style=none] (8.center) to (2);
		\draw (0.center) to (9.center);
		\draw (0.center) to (11.center);
		\draw (13.center) to (12.center);
		\draw [style=none, bend left=45, looseness=1.25] (24.center) to (25.center);
		\draw [style=none, bend right=45, looseness=1.25] (21.center) to (17.center);
		\draw (26.center) to (22.center);
		\draw (16.center) to (19.center);
		\draw (16.center) to (23.center);
		\draw (26.center) to (18.center);
		\draw (3.center) to (29.center);
		\draw [style=none, bend left=45, looseness=1.25] (44.center) to (38.center);
		\draw [style=none, bend right=45, looseness=1.25] (39.center) to (40.center);
		\draw [style=none] (55.center) to (43);
		\draw (36.center) to (56.center);
		\draw (32.center) to (42.center);
		\draw [style=none, bend left=45, looseness=1.25] (31.center) to (37.center);
		\draw [style=none, bend right=45, looseness=1.25] (49.center) to (54.center);
		\draw [style=none] (33.center) to (35);
		\draw (46.center) to (45.center);
		\draw (41.center) to (47.center);
		\draw (41.center) to (34.center);
		\draw (46.center) to (48.center);
	\end{pgfonlayer}
\end{tikzpicture}}

\]
for some phases $V$ and $W$. But then
\[
\scalebox{0.8}{\begin{tikzpicture}
	\begin{pgfonlayer}{nodelayer}
		\node [style=none] (0) at (0, -0) {$=$};
		\node [style=none] (1) at (-2.75, -0.5) {};
		\node [style=none] (2) at (-1.25, -0.5) {};
		\node [style=grey dot] (3) at (-2, -1.25) {};
		\node [style=none] (4) at (-1.25, -0.5) {};
		\node [style=none] (5) at (-2.75, -0.5) {};
		\node [style=none] (6) at (-2, -1.25) {};
		\node [style=none] (7) at (-2, -1.25) {};
		\node [style=none] (8) at (-1.25, -0.5) {};
		\node [style=none] (9) at (-2, -2) {};
		\node [style=none] (10) at (-2.75, 2) {};
		\node [style=grey dot] (11) at (-1.25, -0) {};
		\node [style=corner1] (12) at (-2.75, 1.25) {$V$};
		\node [style=none] (13) at (0.75, 0.5) {};
		\node [style=none] (14) at (1.5, -0.25) {};
		\node [style=grey dot] (15) at (1.5, -0.25) {};
		\node [style=none] (16) at (2.25, 0.75) {};
		\node [style=none] (17) at (2.25, 0.5) {};
		\node [style=none] (18) at (2.25, 0.5) {};
		\node [style=none] (19) at (2.25, 0.5) {};
		\node [style=none] (20) at (0.75, 0.5) {};
		\node [style=none] (21) at (1.5, -0.25) {};
		\node [style=none] (22) at (1.5, -2) {};
		\node [style=none] (23) at (0.75, 0.5) {};
		\node [style=corner1] (24) at (1.5, -1.25) {$W$};
		\node [style=none] (25) at (0.75, 2) {};
		\node [style=grey dot] (26) at (2.25, 0.75) {};
		\node [style=none] (27) at (-5, -2) {};
		\node [style=none] (28) at (-5, 2) {};
		\node [style=box] (29) at (-5, -0) {$V$};
		\node [style=none] (30) at (-4, -0) {$=$};
		\node [style=none] (31) at (4.25, -2) {};
		\node [style=none] (32) at (3, -0) {$=$};
		\node [style=none] (33) at (4.25, 2) {};
		\node [style=box] (34) at (4.25, -0) {$W$};
		\node [style=box] (35) at (-2.75, -0) {$U$};
		\node [style=box] (36) at (0.75, 1) {$U$};
	\end{pgfonlayer}
	\begin{pgfonlayer}{edgelayer}
		\draw [style=none, bend left=45, looseness=1.25] (7.center) to (5.center);
		\draw [style=none, bend right=45, looseness=1.25] (6.center) to (4.center);
		\draw [style=none] (9.center) to (3);
		\draw (1.center) to (10.center);
		\draw (2.center) to (11.center);
		\draw [style=none, bend left=45, looseness=1.25] (14.center) to (20.center);
		\draw [style=none, bend right=45, looseness=1.25] (21.center) to (19.center);
		\draw [style=none] (22.center) to (15);
		\draw (13.center) to (23.center);
		\draw (16.center) to (18.center);
		\draw (16.center) to (26.center);
		\draw (13.center) to (25.center);
		\draw (28.center) to (27.center);
		\draw (33.center) to (31.center);
	\end{pgfonlayer}
\end{tikzpicture}}

\]
yielding the result with $g = f \circ V \circ U$. The statement about morphisms $\obb{I} \tens \obb{A} \to \obb{I} \tens \obb{B}$ follows similarly. Hence by Lemma~\ref{lem:mon_cat_helpful}, $(\tens,\aalpha, \obb{I}, \beta)$ extends to a monoidal structure on $\plusI{\catC}$. Finally from their definitions we quickly see that $[f \tens g]=[f] \otimes [g]$, $[\aalpha]= \alpha$, and $[\beta]=\rho_I$, and hence by~\cite[Proposition~3.5]{kock2008elementary} the functor $[-]$ is stict monoidal.
\end{proof}

\begin{lemma} 
In the situation of Theorem~\ref{thm:constr_is_monoidal}, if $\catC$ is braided or symmetric monoidal then so are $\plusI{\catC}$ and the functor $[-]$.
\end{lemma}
\begin{proof}
Define $\ssigma_{\obb{A}, \obb{B}} \colon \obb{A} \tens \obb{B} \to \obb{B} \tens \obb{A}$ to be the unique map such that 
\[
\sigma_{\obb{A}, \obb{B}} \circ c_{\obb{A},\obb{B}} = c_{\obb{B}, \obb{A}} \circ \ssigma_{\obb{A}, \obb{B}}
\]
again establishing existence with Lemma~\ref{lem:useful}. Since $\sigma_{\obb{A}, \obb{B}}$ is an isomorphism (with inverse $\sigma_{\obb{B}, \obb{A}}$ in the symmetric case), uniqueness forces $\ssigma_{\obb{A}, \obb{B}}$ to be the same. Naturality of $\ssigma$ is easily verified using monicity of the $c_{\obb{A}, \obb{B}}$ and the definition of $\tens$.
We now check the first hexagon equation, with the second being shown dually. 
\[
\scalebox{0.8}{\input{./figures/sup29l.tikz}}

\]
\end{proof}

To show next that $\plusI{\catC}$ has coproducts, we use the following. 

\begin{lemma} \label{lem:deleter_cancellation_modified}
Let $\catC$ be monoidal with distributive finite phased coproducts. Then $I$ is a phase generator. 
\end{lemma}

\begin{proof}
Let $\obb{I} = I \pcoprod I$, $\tinycounit \colon \obb{I} \to I$ with $\tinycounit \circ \pcoproj_1 = \tinycounit \circ \pcoproj_2 = \id{I}$ and $U$ be a phase on $\obb{I}$ with $\tinycounit \circ U = \tinycounit$. We need to show that $U = \id{\obb{I}}$. Let $\tinymultflip \colon \obb{I} \to \obb{I} \otimes \obb{I}$ with $\tinymultflip \circ \pcoproj_i =  (\pcoproj_i \otimes \pcoproj_i) \circ \rho_I^{-1}$ for $i = 1, 2$. Applying the $\pcoproj_i$ we see that there are phases $Q, V$ and $W$ with 
\[
\scalebox{0.8}{\begin{tikzpicture}
	\begin{pgfonlayer}{nodelayer}
		\node [style=none] (0) at (1.5, 0.5) {};
		\node [style=none] (1) at (5.25, -1) {};
		\node [style=none] (2) at (4, -0) {$=$};
		\node [style=corner3] (3) at (1.5, 0.5) {};
		\node [style=none] (4) at (5.25, 1.25) {};
		\node [style=corner1] (5) at (5.25, 0.25) {$V$};
		\node [style=none] (6) at (3, 0.5) {};
		\node [style=none] (7) at (2.25, -1) {};
		\node [style=none] (8) at (3, 0.5) {};
		\node [style=none] (9) at (1.5, 0.5) {};
		\node [style=none] (10) at (2.25, -0.25) {};
		\node [style=none] (11) at (2.25, -0.25) {};
		\node [style=none] (12) at (1.5, 0.5) {};
		\node [style=none] (13) at (3, 0.5) {};
		\node [style=corner3] (14) at (2.25, -0.25) {};
		\node [style=none] (15) at (1.5, 0.5) {};
		\node [style=none] (16) at (-0.75, -1) {};
		\node [style=none] (17) at (-3, 0.5) {};
		\node [style=none] (18) at (-0.75, 1.25) {};
		\node [style=corner3] (19) at (-3.75, -0.25) {};
		\node [style=none] (20) at (-3.75, -0.25) {};
		\node [style=corner3] (21) at (-3, 0.5) {};
		\node [style=none] (22) at (-3, 0.5) {};
		\node [style=none] (23) at (-4.5, 0.5) {};
		\node [style=none] (24) at (-3.75, -0.25) {};
		\node [style=none] (25) at (-4.5, 0.5) {};
		\node [style=none] (26) at (-2, -0) {$=$};
		\node [style=none] (27) at (-3, 0.5) {};
		\node [style=none] (28) at (-4.5, 0.5) {};
		\node [style=none] (29) at (-3.75, -1) {};
		\node [style=none] (30) at (-3, 0.5) {};
		\node [style=corner1] (31) at (-0.75, 0.25) {$Q$};
		\node [style=none] (32) at (9.5, 1.5) {};
		\node [style=none] (33) at (8, 0.5) {};
		\node [style=none] (34) at (8.75, -0.25) {};
		\node [style=corner1] (35) at (12, -0.25) {$W$};
		\node [style=none] (36) at (8, 0.5) {};
		\node [style=none] (37) at (10.5, -0) {$=$};
		\node [style=none] (38) at (8.75, -0.25) {};
		\node [style=none] (39) at (8, 0.5) {};
		\node [style=none] (40) at (12, -1) {};
		\node [style=corner3] (41) at (8.75, -0.25) {};
		\node [style=none] (42) at (12, 0.75) {};
		\node [style=none] (43) at (8.75, -1) {};
		\node [style=none] (44) at (9.5, 0.5) {};
		\node [style=none] (45) at (8, 1.5) {};
		\node [style=corner1] (46) at (8, 0.75) {$U$};
		\node [style=none] (47) at (12.75, 1.5) {};
		\node [style=none] (48) at (12, 0.75) {};
		\node [style=none] (49) at (11.25, 1.5) {};
		\node [style=corner3] (50) at (12, 0.75) {};
		\node [style=none] (51) at (11.25, 1.5) {};
		\node [style=none] (52) at (11.25, 1.5) {};
		\node [style=none] (53) at (12, 0.75) {};
		\node [style=none] (54) at (3, 1) {};
		\node [style=none] (55) at (-4.5, 1) {};
	\end{pgfonlayer}
	\begin{pgfonlayer}{edgelayer}
		\draw [style=none] (4.center) to (1.center);
		\draw [style=none, bend left=45, looseness=1.25] (10.center) to (9.center);
		\draw [style=none, bend right=45, looseness=1.25] (11.center) to (8.center);
		\draw [style=none] (7.center) to (14);
		\draw (12.center) to (15.center);
		\draw [style=none] (18.center) to (16.center);
		\draw [style=none, bend right=45, looseness=1.25] (24.center) to (30.center);
		\draw [style=none, bend left=45, looseness=1.25] (20.center) to (25.center);
		\draw [style=none] (29.center) to (19);
		\draw (22.center) to (17.center);
		\draw [style=none] (42.center) to (40.center);
		\draw [style=none, bend left=45, looseness=1.25] (38.center) to (36.center);
		\draw [style=none] (43.center) to (41);
		\draw [style=none, bend right=45, looseness=1.25] (34.center) to (44.center);
		\draw [style=none] (33.center) to (45.center);
		\draw [style=none] (44.center) to (32.center);
		\draw [style=none, bend left=45, looseness=1.25] (48.center) to (49.center);
		\draw [style=none, bend right=45, looseness=1.25] (53.center) to (47.center);
		\draw [style=none] (23.center) to (55.center);
		\draw [style=none] (6.center) to (54.center);
	\end{pgfonlayer}
\end{tikzpicture}}

\]
But then
\[
\scalebox{0.8}{\begin{tikzpicture}
	\begin{pgfonlayer}{nodelayer}
		\node [style=none] (0) at (-3, 0.5) {};
		\node [style=none] (1) at (-5, -1.25) {};
		\node [style=none] (2) at (-3.75, -0) {$=$};
		\node [style=corner3] (3) at (-3, 0.5) {};
		\node [style=none] (4) at (-5, 1.25) {};
		\node [style=corner1] (5) at (-5, -0) {$V$};
		\node [style=none] (6) at (-1.5, 0.5) {};
		\node [style=none] (7) at (-2.25, -1.25) {};
		\node [style=none] (8) at (-1.5, 0.5) {};
		\node [style=none] (9) at (-3, 0.5) {};
		\node [style=none] (10) at (-2.25, -0.25) {};
		\node [style=none] (11) at (-2.25, -0.25) {};
		\node [style=none] (12) at (-3, 0.5) {};
		\node [style=none] (13) at (-1.5, 0.5) {};
		\node [style=corner3] (14) at (-2.25, -0.25) {};
		\node [style=none] (15) at (-3, 0.5) {};
		\node [style=none] (16) at (8.75, -1) {};
		\node [style=none] (17) at (8.75, 2) {};
		\node [style=none] (18) at (3.5, -0) {$=$};
		\node [style=corner1] (19) at (8.75, 1) {$V$};
		\node [style=none] (20) at (-0.5, -0) {$=$};
		\node [style=none] (21) at (-1.5, 1.25) {};
		\node [style=none] (22) at (7, -0) {$=$};
		\node [style=none] (23) at (0.75, 1.5) {};
		\node [style=corner3] (24) at (1.5, -0.5) {};
		\node [style=none] (25) at (2.25, 1.25) {};
		\node [style=none] (26) at (1.5, -0.5) {};
		\node [style=corner3] (27) at (0.75, 1.5) {};
		\node [style=none] (28) at (0.75, 0.25) {};
		\node [style=none] (29) at (2.25, 0.25) {};
		\node [style=none] (30) at (1.5, -0.5) {};
		\node [style=none] (31) at (2.25, 0.25) {};
		\node [style=none] (32) at (0.75, 0.25) {};
		\node [style=none] (33) at (2.25, 0.25) {};
		\node [style=none] (34) at (0.75, 0.25) {};
		\node [style=none] (35) at (1.5, -1.25) {};
		\node [style=box] (36) at (0.75, 0.5) {$U$};
		\node [style=none] (37) at (4.5, 1.5) {};
		\node [style=corner3] (38) at (5.25, 0.75) {};
		\node [style=none] (39) at (5.25, 0.75) {};
		\node [style=none] (40) at (6, 2) {};
		\node [style=none] (41) at (4.5, 1.5) {};
		\node [style=corner3] (42) at (4.5, 1.5) {};
		\node [style=none] (43) at (6, 1.5) {};
		\node [style=none] (44) at (5.25, 0.75) {};
		\node [style=none] (45) at (6, 1.5) {};
		\node [style=none] (46) at (4.5, 1.5) {};
		\node [style=none] (47) at (6, 1.5) {};
		\node [style=none] (48) at (5.25, -1) {};
		\node [style=none] (49) at (4.5, 1.5) {};
		\node [style=box] (50) at (5.25, -0.25) {$W$};
		\node [style=box] (51) at (8.75, -0.25) {$W$};
	\end{pgfonlayer}
	\begin{pgfonlayer}{edgelayer}
		\draw [style=none] (4.center) to (1.center);
		\draw [style=none, bend left=45, looseness=1.25] (10.center) to (9.center);
		\draw [style=none, bend right=45, looseness=1.25] (11.center) to (8.center);
		\draw [style=none] (7.center) to (14);
		\draw (12.center) to (15.center);
		\draw [style=none] (17.center) to (16.center);
		\draw [style=none] (6.center) to (21.center);
		\draw [style=none, bend left=45, looseness=1.25] (30.center) to (34.center);
		\draw [style=none, bend right=45, looseness=1.25] (26.center) to (31.center);
		\draw [style=none] (35.center) to (24);
		\draw (28.center) to (23.center);
		\draw [style=none] (29.center) to (25.center);
		\draw [style=none, bend left=45, looseness=1.25] (44.center) to (49.center);
		\draw [style=none, bend right=45, looseness=1.25] (39.center) to (45.center);
		\draw [style=none] (48.center) to (38);
		\draw (41.center) to (37.center);
		\draw [style=none] (43.center) to (40.center);
	\end{pgfonlayer}
\end{tikzpicture}}

\]
and so $W = \id{}$. Hence $U \circ Q = Q \circ W = Q$ and so $U = \id{}$. 

For the next property, let $m \colon \obb{I} \to A \pcoprod B$ be a diagonal monomorphism and $U$ a phase on $A \pcoprod B$ with $U \circ m = m$. We need to show that $U = \id{A \pcoprod B}$. Let $\tinymultflip[whitedot] \colon A \pcoprod B \to (A \pcoprod B) \otimes \obb{I}$ with $\tinymultflip[whitedot] \circ \pcoproj_A = (\pcoproj_A \circ \pcoproj_1) \circ {\rho_A}^{-1}$ and $\tinymultflip[whitedot] \circ \pcoproj_B = (\pcoproj_B \circ \pcoproj_2) \circ {\rho_B}^{-1}$. Applying coprojections and using distributivity we see that there are phases $V$ and $W$ on $\obb{I}$ with
\[
\scalebox{0.8}{\begin{tikzpicture}
	\begin{pgfonlayer}{nodelayer}
		\node [style=none] (0) at (9.5, 1.75) {};
		\node [style=none] (1) at (8, 0.5) {};
		\node [style=none] (2) at (8.75, -0.25) {};
		\node [style=corner1] (3) at (12, -0.5) {$W$};
		\node [style=none] (4) at (8, 0.5) {};
		\node [style=none] (5) at (10.5, -0) {$=$};
		\node [style=none] (6) at (8.75, -0.25) {};
		\node [style=none] (7) at (8, 0.5) {};
		\node [style=none] (8) at (12, -1.25) {};
		\node [style=corner3] (9) at (8.75, -0.25) {};
		\node [style=none] (10) at (8.75, -1.25) {};
		\node [style=none] (11) at (9.5, 0.5) {};
		\node [style=none] (12) at (8, 1.75) {};
		\node [style=corner1] (13) at (8, 0.75) {$m$};
		\node [style=none] (14) at (12.75, 2.5) {};
		\node [style=none] (15) at (12, 1.75) {};
		\node [style=none] (16) at (11.25, 2.5) {};
		\node [style=white dot] (17) at (12, 1.75) {};
		\node [style=none] (18) at (11.25, 2.5) {};
		\node [style=none] (19) at (11.25, 2.5) {};
		\node [style=none] (20) at (12, 1.75) {};
		\node [style=box] (21) at (12, 0.75) {$m$};
		\node [style=none] (22) at (-4.5, 1.75) {};
		\node [style=none] (23) at (-3.75, 1) {};
		\node [style=none] (24) at (-3.75, -1.25) {};
		\node [style=none] (25) at (-3, 1.75) {};
		\node [style=none] (26) at (-4.5, 1.75) {};
		\node [style=none] (27) at (-1, 0.5) {};
		\node [style=none] (28) at (-0.25, -0.25) {};
		\node [style=none] (29) at (-4.5, 1.75) {};
		\node [style=white dot] (30) at (-3.75, 1) {};
		\node [style=none] (31) at (-0.25, -0.25) {};
		\node [style=none] (32) at (-1, 1.75) {};
		\node [style=none] (33) at (-3.75, 1) {};
		\node [style=none] (34) at (-2, -0) {$=$};
		\node [style=none] (35) at (0.5, 0.5) {};
		\node [style=none] (36) at (-1, 0.5) {};
		\node [style=none] (37) at (-0.25, -1.25) {};
		\node [style=none] (38) at (-1, 0.5) {};
		\node [style=none] (39) at (0.5, 1.75) {};
		\node [style=box] (40) at (-3.75, -0) {$U$};
		\node [style=white dot] (41) at (-0.25, -0.25) {};
		\node [style=box] (42) at (0.5, 1) {$V$};
	\end{pgfonlayer}
	\begin{pgfonlayer}{edgelayer}
		\draw [style=none, bend left=45, looseness=1.25] (6.center) to (4.center);
		\draw [style=none] (10.center) to (9);
		\draw [style=none, bend right=45, looseness=1.25] (2.center) to (11.center);
		\draw [style=none] (1.center) to (12.center);
		\draw [style=none] (11.center) to (0.center);
		\draw [style=none, bend left=45, looseness=1.25] (15.center) to (16.center);
		\draw [style=none, bend right=45, looseness=1.25] (20.center) to (14.center);
		\draw [style=none] (8.center) to (20.center);
		\draw [style=none, bend left=45, looseness=1.25] (31.center) to (38.center);
		\draw [style=none] (37.center) to (41);
		\draw [style=none, bend right=45, looseness=1.25] (28.center) to (35.center);
		\draw [style=none] (36.center) to (32.center);
		\draw [style=none] (35.center) to (39.center);
		\draw [style=none, bend left=45, looseness=1.25] (33.center) to (29.center);
		\draw [style=none, bend right=45, looseness=1.25] (23.center) to (25.center);
		\draw [style=none] (24.center) to (23.center);
	\end{pgfonlayer}
\end{tikzpicture}}

\]
Then we have
\[
\scalebox{0.8}{\begin{tikzpicture}
	\begin{pgfonlayer}{nodelayer}
		\node [style=none] (0) at (-10.25, 1.25) {};
		\node [style=none] (1) at (-11.75, -0) {};
		\node [style=none] (2) at (-11, -0.75) {};
		\node [style=corner1] (3) at (-7, -1) {$W$};
		\node [style=none] (4) at (-11.75, -0) {};
		\node [style=none] (5) at (-9, -0) {$=$};
		\node [style=none] (6) at (-11, -0.75) {};
		\node [style=none] (7) at (-11.75, -0) {};
		\node [style=none] (8) at (-7, -1.75) {};
		\node [style=corner3] (9) at (-11, -0.75) {};
		\node [style=none] (10) at (-11, -1.75) {};
		\node [style=none] (11) at (-10.25, -0) {};
		\node [style=none] (12) at (-11.75, 1.25) {};
		\node [style=corner1] (13) at (-11.75, 0.25) {$m$};
		\node [style=none] (14) at (-6.25, 2) {};
		\node [style=none] (15) at (-7, 1.25) {};
		\node [style=none] (16) at (-7.75, 2) {};
		\node [style=white dot] (17) at (-7, 1.25) {};
		\node [style=none] (18) at (-7.75, 2) {};
		\node [style=none] (19) at (-7.75, 2) {};
		\node [style=none] (20) at (-7, 1.25) {};
		\node [style=box] (21) at (-7, 0.25) {$m$};
		\node [style=none] (22) at (0, 2) {};
		\node [style=none] (23) at (0.75, 1.25) {};
		\node [style=none] (24) at (0.75, 1.25) {};
		\node [style=none] (25) at (0, 3.25) {};
		\node [style=none] (26) at (1.5, 2) {};
		\node [style=none] (27) at (0, 2) {};
		\node [style=none] (28) at (0.75, 1) {};
		\node [style=none] (29) at (0, 2) {};
		\node [style=none] (30) at (1.5, 3.25) {};
		\node [style=white dot] (31) at (0.75, 1.25) {};
		\node [style=box] (32) at (1.5, 2.5) {$V$};
		\node [style=none] (33) at (-4.25, 3.25) {};
		\node [style=none] (34) at (-3.5, -1.75) {};
		\node [style=none] (35) at (-5, -0) {$=$};
		\node [style=box] (36) at (-3.5, 0.25) {$m$};
		\node [style=none] (37) at (-2.75, 3.25) {};
		\node [style=none] (38) at (-3.5, 2.5) {};
		\node [style=none] (39) at (-4.25, 3.25) {};
		\node [style=none] (40) at (-3.5, 2.5) {};
		\node [style=corner1] (41) at (-3.5, -1) {$W$};
		\node [style=none] (42) at (-4.25, 3.25) {};
		\node [style=white dot] (43) at (-3.5, 2.5) {};
		\node [style=none] (44) at (-1.5, -0) {$=$};
		\node [style=box] (45) at (-3.5, 1.5) {$U$};
		\node [style=none] (46) at (0.75, -1.75) {};
		\node [style=corner1] (47) at (0.75, -1) {$W$};
		\node [style=box] (48) at (0.75, 0.25) {$m$};
		\node [style=none] (49) at (3, -0) {$=$};
		\node [style=none] (50) at (6.25, 1.75) {};
		\node [style=corner1] (51) at (4.75, 0.75) {$m$};
		\node [style=none] (52) at (5.5, -0.75) {};
		\node [style=none] (53) at (5.5, -1.75) {};
		\node [style=none] (54) at (4.75, -0) {};
		\node [style=none] (55) at (6.25, -0) {};
		\node [style=corner3] (56) at (5.5, -0.75) {};
		\node [style=none] (57) at (5.5, -0.75) {};
		\node [style=none] (58) at (4.75, 1.75) {};
		\node [style=none] (59) at (4.75, -0) {};
		\node [style=none] (60) at (4.75, -0) {};
		\node [style=corner1] (61) at (6.25, 0.75) {$V$};
	\end{pgfonlayer}
	\begin{pgfonlayer}{edgelayer}
		\draw [style=none, bend left=45, looseness=1.25] (6.center) to (4.center);
		\draw [style=none] (10.center) to (9);
		\draw [style=none, bend right=45, looseness=1.25] (2.center) to (11.center);
		\draw [style=none] (1.center) to (12.center);
		\draw [style=none] (11.center) to (0.center);
		\draw [style=none, bend left=45, looseness=1.25] (15.center) to (16.center);
		\draw [style=none, bend right=45, looseness=1.25] (20.center) to (14.center);
		\draw [style=none] (8.center) to (20.center);
		\draw [style=none, bend left=45, looseness=1.25] (24.center) to (29.center);
		\draw [style=none] (28.center) to (31);
		\draw [style=none, bend right=45, looseness=1.25] (23.center) to (26.center);
		\draw [style=none] (27.center) to (25.center);
		\draw [style=none] (26.center) to (30.center);
		\draw [style=none, bend left=45, looseness=1.25] (38.center) to (42.center);
		\draw [style=none, bend right=45, looseness=1.25] (40.center) to (37.center);
		\draw [style=none] (34.center) to (40.center);
		\draw [style=none] (28.center) to (46.center);
		\draw [style=none, bend left=45, looseness=1.25] (57.center) to (60.center);
		\draw [style=none] (53.center) to (56);
		\draw [style=none, bend right=45, looseness=1.25] (52.center) to (55.center);
		\draw [style=none] (54.center) to (58.center);
		\draw [style=none] (55.center) to (50.center);
	\end{pgfonlayer}
\end{tikzpicture}}

\]
and so composing with $\tinycounit$ and using monicity of $m$ we obtain
\[
\scalebox{0.8}{\begin{tikzpicture}
	\begin{pgfonlayer}{nodelayer}
		\node [style=none] (0) at (3, -0.5) {$=$};
		\node [style=none] (1) at (1.25, 1.5) {};
		\node [style=none] (2) at (0.5, -0.75) {};
		\node [style=none] (3) at (0.5, -1.75) {};
		\node [style=none] (4) at (-0.25, -0) {};
		\node [style=none] (5) at (1.25, -0) {};
		\node [style=corner3] (6) at (0.5, -0.75) {};
		\node [style=none] (7) at (0.5, -0.75) {};
		\node [style=none] (8) at (-0.25, -0) {};
		\node [style=none] (9) at (-0.25, -0) {};
		\node [style=corner1] (10) at (1.25, 0.5) {$V$};
		\node [style=none] (11) at (-0.25, 1) {};
		\node [style=grey dot] (12) at (1.25, 1.5) {};
		\node [style=none] (13) at (4.5, -1.75) {};
		\node [style=none] (14) at (4.5, 1) {};
	\end{pgfonlayer}
	\begin{pgfonlayer}{edgelayer}
		\draw [style=none, bend left=45, looseness=1.25] (7.center) to (9.center);
		\draw [style=none] (3.center) to (6);
		\draw [style=none, bend right=45, looseness=1.25] (2.center) to (5.center);
		\draw [style=none] (5.center) to (1.center);
		\draw [style=none] (4.center) to (11.center);
		\draw [style=none] (14.center) to (13.center);
	\end{pgfonlayer}
\end{tikzpicture}}

\]
But now $(\tinycounit \otimes \id{\obb{I}}) \circ \tinymultflip$ is a phase and so is epic. Hence by the first part we have $V = \id{\obb{I}}$. Similarly $(\id{A \pcoprod B} \otimes \tinycounit) \circ \tinymultflip[whitedot] = Q$ for some phase $Q$ on $A \pcoprod B$, giving $Q \circ U = Q$ and so $U = \id{}$. 
\end{proof}

\begin{theorem} \label{thm:getmoncoprod}
Let $\catC$ be a monoidal category with distributive monic finite phased coproducts with transitive phases. Then $\plusI{\catC}$ has distributive, monic finite coproducts.  
\end{theorem}
\begin{proof}
The monoidal unit $I$ is a phase generator by Lemma~\ref{lem:deleter_cancellation_modified}. Hence by Theorem~\ref{thm:localToGlobal} $\plusI{\ctb}$ has finite coproducts $\obb{A} + \obb{B}$ and these are sent by $[-]$ to phased coproducts in $\ctb$. 
For distributivity consider the unique $f \colon \obb{A} \tens \obb{C} + \obb{B} \tens \obb{C} \to (\obb{A} + \obb{B}) \tens \obb{C}$ in $\plusI{\ctb}$ with $f \circ \pcoproj_1 = \coproj_{\obb{A}} \tens \id{\obb{C}}$ and $f \circ \pcoproj_2 =  \coproj_{\obb{B}} \tens \id{\obb{C}}$. We have
\[
[f] \circ \pcoproj_{A \otimes C} 
=
[f] \circ [\coproj_{\obb{A} \tens \obb{C}}]
=
[\coproj_{\obb{A}} \tens \id{\obb{C}}]
=
[\coproj_{\obb{A}}] \otimes [\id{\obb{C}}]
=
\pcoproj_A \otimes \id{C}
\]
and $[f] \circ \pcoproj_{B \otimes C} = \pcoproj_B \otimes \id{C}$ also. By distributivity in $\ctb$, $[f]$ is then an isomorphism. But since phases are invertible, $[-]$ reflects isomorphisms, so $f$ is invertible. 
\end{proof}

To equip $\plusI{\catC}$ with a choice of global phases we will use the following.

\begin{lemma} \label{lem:central-helpful}
In any monoidal category with distributive monic finite coproducts a scalar $s$ is central iff for every object $A$ there is a scalar $t$ with $s \cdot \id{A} = \id{A} \cdot t$.
\end{lemma}
\begin{proof}
Let $A$ be any object. Suppose that $s \cdot \id{A + I} = \id{A + I} \cdot t$ for some scalar $t$. Then $\coproj_I \circ s = \coproj_I \circ t$ and so by monicity of $\pcoproj_I$ we have $s = t$. But then $\coproj_A \circ (s \cdot \id{A}) = \coproj_A \circ (\id{A} \cdot s)$ and so by monicity again $s \cdot \id{A} = \id{A} \cdot s$. 
\end{proof}

\begin{lemma} \label{lem:phasesAreScalars}
Let $\catC$ be a monoidal category with distributive monic phased coproducts with transitive phases. Then $\plusI{\catC}$ has a canonical choice of global phases
\[
\mathbb{P} := \{u \colon \obb{I} \to \obb{I} \mid u \text{ is a phase on $\obb{I}$ in $\catC$} \}
\]
where $\obb{I} = I \pcoprod I$ is its monoidal unit.
Moreover, phases $U$ on $\obb{A} = A \pcoprod I$ in $\catC$ are precisely morphisms in $\plusI{\catC}$ of the form $u \cdot \id{\obb{A}}$ for some $u \in \mathbb{P}$.
\end{lemma}

\begin{proof}
We begin with the second statement. 
An endomorphism $U$ of $\obb{A}$ in $\plusI{\ctb}$ is a phase on $\obb{A}$ in $\ctb$ iff $[U] = \id{A}$.
For any $u$ as above, since $[-]$ is strict monoidal we indeed have $[u \cdot \id{\obb{A}}] = [u] \cdot [\id{\obb{A}}] = \id{A}$, and so $u \cdot \id{\obb{A}}$ is a phase. 

Conversely, for any phase $U$ on $\obb{A}$, consider it instead as an automorphism $V$ of $\obb{A} \tens \obb{I}$ in $\plusI{\ctb}$. Then $[V] = \id{A}$, and so $V$ is a phase on $\obb{A} \tens \obb{I}$. 
Now in $\ctb$, by distributivity, $\obb{A} \otimes \obb{I}$ forms $A \otimes \obb{I} \pcoprod I \otimes \obb{I}$ with every phase being of the form $\id{\obb{A}} \otimes u$ for some phase $u$ on $\obb{I}$. Moreover, $c:= c_{\obb{A}, \obb{I}} \colon \obb{A} \tens \obb{I} \to \obb{A} \otimes \obb{I}$ is then diagonal as a morphism from $\obb{A} \tens \obb{I}$ into this phased coproduct. Hence by transitivity $c \circ V = (\id{\obb{A}} \otimes u) \circ c$ for some phase $u$ of $\obb{I}$. But this states precisely that in $\plusI{\ctb}$ we have $V = \id{\obb{A}} \tens u$ or equivalently $U = u \cdot \id{\obb{A}}$.

Dually, every phase is of the form $\id{\obb{A}} \cdot v$ for some $v \in \mathbb{P}$. In particular for each $u \in \mathbb{P}$ so is $u \cdot \id{\obb{A}}$. Hence by Lemma~\ref{lem:central-helpful} every $u \in \mathbb{P}$ is central, making $\mathbb{P}$ a valid choice of global phases.
\end{proof} 

\begin{corollary} \label{cor:phcoprodcorrespon}
There is a one-to-one correspondence, up to monoidal equivalence, between monoidal categories 
\begin{itemize}
\item $\catC$ with distributive, monic finite phased coproducts with transitive phases;
\item $\catD$ with distributive, monic finite coproducts and choice of global phases $\mathbb{P}$;
\end{itemize}
given by $\catC \mapsto \plusI{\catC}$ and $\catD \mapsto \catD_\quotP$. 
\end{corollary}
\begin{proof}
The assignments are well-defined by Theorems~\ref{thm:constr_is_monoidal} and~\ref{thm:getmoncoprod} and Lemmas~\ref{lem:globaltolocalmonoidal} and~\ref{lem:phasesAreScalars}. Now by Theorem~\ref{thm:localToGlobal}, $[-]$ induces an equivalence $\catC \simeq \quot{\plusI{\catC}}{\sim}$ where $f \sim g$ when $f = g \circ U$ for some phase $U$ in $\catC$. But now this is strict monoidal since $[-]$ is, and by Lemma~\ref{lem:phasesAreScalars} in $\plusI{\catC}$ every such $U$ is of the form $\id{} \cdot u$ for some $u \in \mathbb{P}$. Hence $\quot{\plusI{\catC}}{\sim} = \plusI{\catC}_\quotP$.

Conversely, we must check that $\catD \simeq \plusI{\catD_\quotP}$ for such a category $\catD$. Define a functor $F \colon \catD \to \plusI{\catD_\quotP}$ on objects by $F(A) = A + I$ and for $f \colon A \to B$ by setting $F(f) = [f + \id{I}]_\tp \colon A + I \to B + I$, where $[-]_\tp$ denotes equivalence classes under~\eqref{eq:quotient_rule_general}. By Lemma~\ref{lem:globaltolocalmonoidal} the phased coproducts in $\catD_\quotP$ are precisely the coproducts in $\catD$, making $F$ well-defined. Now every $[g]_\tp \colon F(A) \to F(B)$ in $\plusI{\catD_\quotP}$ has $g = h + u$ for a unique $h \colon A \to B$ and $(u \colon I \to I) \in \mathbb{P}$. Then $[g]_\tp = F(f)$ iff 
\[
(f + \id{I}) = v \cdot (h + u) = (v \cdot h + v \cdot u)
\]
for some $v \in \mathbb{P}$. So $[g]_\tp = F(f)$ for the unique morphism $f = u^{-1} \cdot h$, making $F$ full and faithful. It is essentially surjective on objects by Lemma~\ref{lem:isoms}, and distributivity in $\catD$ ensures that $F$ is strong monoidal. Clearly $F$ also restricts to an isomorphism of global phase groups. 
\end{proof}

This correspondence can be made into an equivalence of categories, giving the $\mathsf{GP}$ construction a universal property; see appendix~\ref{app:universality}.

\begin{examples} \label{ex:ofconstruction}
We've seen that $\VecSp_{}$, $\Hilb$ and $\FVecSp_{k}$ satisfy the above properties of $\catD$ and so they may be reconstructed from their quotients as
\[
\VecSp_{} \simeq \plusI{\VecP} 
\qquad
\Hilb \simeq \plusI{\HilbP}
\qquad
\FVecSp_{k} \simeq \plusI{\VecProj}
\]
\end{examples}

\section{Phased Biproducts} \label{sec:phbiprod}

Recall that a category $\catC$ has \emph{zero arrows} when it comes with a family of morphisms $0  = 0_{A,B} \colon A \to B$ with $f \circ 0 = 0$ and $0 = 0 \circ g$ for all morphisms $f, g$. In a monoidal category we also require $0 \otimes f = 0$ and $g \otimes 0 = 0$. As in Remark~\ref{rem:phasedproducts} a \emph{phased product} $A \pprod B$ is defined via projections $\pproj_A \colon A \pprod B \to A$ and $\pproj_B \colon A \pprod B \to B$ satisfying the dual condition to that of a phased coproduct.

\begin{definition} \label{def:phased_biprod}
In a category with zero arrows, a \emph{phased biproduct} of objects $A, B$ is an object $A \pbiprod B$ together with morphisms
\[
\begin{tikzcd}[row sep = large]
A \rar[shift left = 2.5]{\pcoproj_A}  & A \pbiprod B  \lar[shift left = 2.5]{\pproj_A} \rar[shift right = 2.5,swap]{\pproj_B} & B \lar[shift right = 2.5,swap]{\pcoproj_B}  
\end{tikzcd}
\]
satisfying the equations
\begin{align*}
\pproj_A \circ \pcoproj_A &= \id{A} & \pproj_B \circ \pcoproj_A &= 0\\
\pproj_A \circ \pcoproj_B &= 0 & \pproj_B \circ \pcoproj_B &= \id{B} 
\end{align*}
\noindent
and for which $(\pcoproj_A, \pcoproj_B)$ and $(\pproj_A, \pproj_B)$ make $A \pbiprod B$ a phased coproduct and phased product, respectively, such that each have the same phases $U \colon A \pbiprod B \to A \pbiprod B$. 
\end{definition}

We can define a phased biproduct $A_1 \pbiprod \dots \pbiprod A_n$ of any finite collection of objects similarly. The usual notion of a biproduct~\cite[§~VIII.2]{mac1978categories} is then a phased biproduct whose only phase is the identity. By an empty phased biproduct we mean a \emph{zero object}; an object $0$ with $\id{0} = 0$, or equivalently which is both initial and terminal. 

\begin{lemma} \label{lem:ph_biprod}
Let $\catC$ have a zero object and binary phased biproducts. Then:
\begin{enumerate}[label=\arabic*., ref=\arabic*]
\item \label{enum:hasfinite}
$\catC$ has finite phased biproducts;
\item \label{enum:unique}
Any phased coproduct $A_1 \pcoprod \dots \pcoprod A_n$ in $\catC$ has a unique phased biproduct structure;
\item \label{enum:transitive}
All phases are transitive.
\end{enumerate}
\end{lemma}

\begin{proof}

\ref{enum:hasfinite}. We will show that any object $(A \pbiprod B) \pbiprod C$ forms a phased biproduct of $A$, $B$ and $C$, with the general case of $((A_1 \pbiprod A_2) \pbiprod \cdots ) A_n$ being similar. By Proposition~\ref{prop:assoc} and its dual any such object forms a phased coproduct and product with coprojections $\pcoproj_{A \biprod B} \circ \pcoproj_A$,  
$\pcoproj_{A \biprod B} \circ \pcoproj_B$, and $\pcoproj_C$, and 
projections $\pproj_A \circ \pproj_{A \pbiprod B}$, $\pproj_B \circ \pproj_{A \pbiprod B}$ and $\pproj_C$. It's routine to check that these satisfy the necessary equations.

It remains to check that any endomorphism $U$ of $(A \pbiprod B) \pbiprod C$ preserving these coprojections then preserves the projections, with the converse statement then being dual. In this case we have 
\[
\pproj_{A \pbiprod B} \circ U \circ \pcoproj_{C} = 0 \qquad U \circ \pcoproj_{A \pbiprod B} = \pcoproj_{A \pbiprod B} \circ V\] 
for some phase $V$ on $A \pbiprod B$. But then $\pi_{A \pbiprod B} \circ U$ and $V \circ \pproj_{A \pbiprod B}$ have equal composites with $\pcoproj_{A \pbiprod B}$ and $\pcoproj_C$ and so
\[
 \pi_{A \pbiprod B} \circ U = V \circ \pproj_{A \pbiprod B} \circ W 
 \] 
 for some phase $W$. But then $\pi_{A \pbiprod B} \circ U = V \circ \pproj_{A \pbiprod B}$, ensuring that $U$ preserves the above projections. 

\ref{enum:unique}. We show the result for binary phased coproducts $A \pcoprod B$, with the $n$-ary case being similar. By Lemma~\ref{lem:isoms} any coprojection preserving morphism
\[
\begin{tikzcd}
A \pcoprod B \rar{g} & A \pbiprod B
\end{tikzcd}
\]
  is an isomorphism, and one may then check that $\pproj_A \circ g$ and $\pproj_B \circ g$ form projections making $A \pcoprod B$ a phased biproduct.

 For uniqueness note that for any phased biproduct, any $p_A \colon A \pbiprod B \to A$ with $p_A \circ \coproj_A = \id{A}$ and $p_A \circ \pcoproj_B = 0$ has $p_A = \pi_A \circ U$ for some phase $U$. But $\pi_A \circ U = \pi_A$ and so $p_A = \pi_A$ is unique. 

\ref{enum:transitive}. For any diagonal morphism 
 \[
\begin{tikzcd}
 A \pcoprod B \rar{f} &  C \pcoprod D
\end{tikzcd}
 \]
with $f \circ \pcoproj_A = \pcoproj_C \circ g$, by composing with the coprojections, we see that the unique projections $\pproj_C$ and $\pproj_A$ have $\pproj_C \circ f = g \circ \pproj_A$. Then for any phase $U$ we have
\[
\pproj_C \circ f \circ U 
= 
g \circ \pproj_A \circ U 
=
g \circ \pproj_A
= 
\pproj_C \circ f 
\] 
and $\pproj_D \circ f \circ U = \pproj_D \circ f$ also. Hence $f \circ U = V \circ f$ for some phase $V$ on $C \pcoprod D$. 
\end{proof}

In a category with phased biproducts, by a \indef{phase generator} let us now mean an object satisfying the properties of Definition~\ref{def:phase-gen} along with the dual statements about phased products. 

\begin{lemma} \label{lem:getBiprod}
Let $\catC$ be a category with finite phased biproducts with a phase generator $I$. Then $\plusI{\catC}$ has finite biproducts. Conversely, if $\catD$ has finite biproducts and a transitive choice of trivial isomorphisms then $\quot{\catD}{\sim}$ has finite phased biproducts. 
\end{lemma}
\begin{proof}
Since $\ctb$ has phased biproducts, any phased coproduct $\obb{A} = A \pcoprod I$ has a unique phased biproduct structure $\pproj_A \colon \obb{A} \to A, \pproj_I \colon \obb{A} \to I$ in $\ctb$, and so we may equivalently view the objects of $\plusI{\ctb}$ as such phased biproducts.  Then $\plusI{\ctb}$ has zero morphisms
\[
0_{\obb{A}, \obb{B}} :=
\begin{tikzcd}
\obb{A} \rar{\pproj_I} & I \rar{\pcoproj_I} & \obb{B}
\end{tikzcd}
\]
and in particular the initial object $\obb{0} = 0 \pcoprod I$ has $\id{\obb{0}} = 0$ and so is a zero object.

Now by Theorem~\ref{thm:localToGlobal} for any objects $\obb{A}, \obb{B} \in \plusI{\ctb}$, any object and morphisms 
\[
\begin{tikzcd}[row sep = large]
\obb{A} 
\rar[shift left = 2.5]{\coproj_{\obb{A}}}  
& 
\obb{A \biprod B}  
\lar[shift left = 2.5]{\pi_{\obb{A}}} 
\rar[shift right = 2.5,swap]{\pi_\obb{B}} 
& \obb{B} 
\lar[shift right = 2.5,swap]{\coproj_\obb{B}}  
\end{tikzcd}
\]
which are sent by $[-]$ to a phased biproduct structure on $A, B$ in $\ctb$ have that $\coproj_{\obb{A}}$ and $\coproj_{\obb{B}}$ form a coproduct of $\obb{A}, \obb{B}$ in $\plusI{\ctb}$, and dually $\pproj_{\obb{A}}$ and $\pproj_{\obb{B}}$ form a product. 
Then since $[-]$ reflects zeroes and $[\pi_{\obb{B}} \circ \coproj_{\obb{A}}] = 0$ we have $\pi_{\obb{B}} \circ \coproj_{\obb{A}} = 0_{\obb{A}, \obb{B}}$, and $\pi_{\obb{A}} \circ \coproj_{\obb{B}} = 0_{\obb{B}, \obb{A}}$ similarly. By applying $[-]$ we also see that $\pi_{\obb{A}} \circ \coproj_{\obb{A}} = U$ and $\pproj_{\obb{B}} \circ \coproj_{\obb{B}} = V$ for some phases $U$ on $\obb{A}$ and $V$ on $\obb{B}$. Then finally $\coproj_{\obb{A}}$, $\coproj_{\obb{B}}$, $U^{-1} \circ \pi_\obb{A}$ and $V^{-1} \circ \pi_{\obb{B}}$ make $\obb{A \biprod B}$ a biproduct in $\plusI{\ctb}$.

For the converse statement, we know that biproducts in $\cta$ form distributive phased coproducts in $\quot{\cta}{\sim}$, and dually they form phased products also. The zero arrows in $\cta$ form zero arrows in $\quot{\cta}{\sim}$ with $[f]_\tc = 0 \implies f = 0$. Hence $[-]_\tc$ preserves the phased biproduct equations. Now, endomorphisms on $A \pbiprod B$ in $\quot{\cta}{\sim}$ preserving the coprojections are (equivalence classes) of endomorphisms $U$ of $A \biprod B$ in $\cta$ of the form $U = s + t$ for some $s \in \mathbb{T}_A$ and $t \in \mathbb{T}_B$. But equivalently $U = s \times t$ and so they preserve the projections in $\quot{\cta}{\sim}$.
\end{proof}

In a monoidal category we say that phased biproducts are \emph{distributive} when they are distributive as phased coproducts. 

\begin{corollary} \label{cor:biproducts}
The assignments $\catC \mapsto \plusI{\catC}$ and $\catD \mapsto \catD_\quotP$ give a one-to-one correspondence, up to monoidal equivalence, between monoidal categories 
\begin{itemize}
\item $\catC$ with finite distributive phased biproducts;
\item $\catD$ with finite distributive biproducts and a choice of global phases $\mathbb{P}$.
\end{itemize}
\end{corollary}
\begin{proof}
For such a category $\catC$, $I$ is a generator for phased coproducts by Lemma~\ref{lem:deleter_cancellation_modified} and hence also one for phased biproducts dually. Hence by Lemma~\ref{lem:getBiprod} $\plusI{\catC}$ has finite biproducts. The assignment is then well-defined by Lemma~\ref{lem:getBiprod} and Corollary~\ref{cor:phcoprodcorrespon}.
\end{proof}

\begin{example}
Since $\VecSp{}$, $\Hilb$ and $\FVecSp_k$ all have biproducts these become phased biproducts in $\VecP$, $\HilbP$ and $\VecProj$. 
\end{example}

\section{Compact Categories} \label{sec:compact_cats}

Recall that an object $A$ in a monoidal category has a \emph{dual object} when there is some object $A^*$ and morphisms $\eta \colon I \to A^* \otimes A$ and $\varepsilon \colon A \otimes A^* \to I$, depicted  as $\tinycup$ and $\tinycap$, satisfying the \emph{snake equations}:
\[
  \begin{pic}[scale=.5]
    \draw[arrow=.36, arrow=.66] (0,0) to (0,1) to[out=90,in=90,looseness=2] (1,1) to[out=-90,in=-90,looseness=2] (2,1) to (2,2);
  \end{pic}
  =
  \begin{pic}[scale=.5]
  \draw[arrow=.5] (0,0) to (0,2);
  \end{pic}
  \qquad \qquad
  \begin{pic}[scale=.5]
    \draw[reverse arrow=.37, reverse arrow=.67] (0,0) to (0,1) to[out=90,in=90,looseness=2] (-1,1) to[out=-90,in=-90,looseness=2] (-2,1) to (-2,2);
  \end{pic}
  =
  \begin{pic}[scale=.5]
  \draw[reverse arrow=.5] (0,0) to (0,2);
  \end{pic}
\]
Here we label $A$ and $A^*$ with upward and downward directed arrows, respectively. A monoidal category is \emph{compact closed} when every object has a dual~\cite{kelly1980coherence}. 

\begin{example}
An object in $\VecSp_{}$, $\VecP$, $\Hilb$ or $\HilbP$ has a dual precisely when it is finite-dimensional as a vector space. Restricting to finite dimensions gives respective full subcategories $\FVecSp_{}$, $\FVecP$, $\FHilb$ and $\FHilbP$ which are all compact closed.
\end{example}

In this setting, phased coproducts get several nice properties for free.

\begin{lemma} \label{lem:comp-closed}
Let $\ctb$ be a compact category with finite phased coproducts.  
\begin{enumerate}[label=\arabic*., ref=\arabic*]
\item \label{enum:pinit-pterm}
Any initial object in $\ctb$ is a zero object.
\item \label{enum:distmonic}
Phased coproducts are distributive and monic in $\ctb$.
\item \label{enum:phccompclosed}
$\plusI{\ctb}$ is compact closed.
\item \label{enum:pharescalars} 
Every phase $U$ on $\obb{A}$ is of the form $U = u \cdot \id{\obb{A}}$ in $\plusI{\ctb}$, for some global phase $u$.
\end{enumerate}
\end{lemma}
\begin{proof}
\ref{enum:pinit-pterm}. This is well-known; since $\ctb$ is self-dual it has a terminal object $1$, but since $0 \otimes (-)$ preserves products $1 \simeq 0 \otimes 1$, and also $0 \otimes 1 \simeq 0$ dually.

\ref{enum:distmonic}. The presence of zero arrows makes all coprojections split monic. Now for any phased coproduct $B \pcoprod C$, one may use the bijection on morphisms
\[
\scalebox{0.8}{\begin{tikzpicture}
	\begin{pgfonlayer}{nodelayer}
		\node [style=none] (0) at (-5.25, 1.5) {};
		\node [style=medium box] (1) at (-5.25, -0) {$f$  };
		\node [style=none] (2) at (-6, -1.5) {};
		\node [style=none] (3) at (-4.5, -1.5) {};
		\node [style=none] (4) at (-4.5, -0.25) {};
		\node [style=none] (5) at (-6, -0.25) {};
		\node [style=none] (6) at (-5.25, 0.5) {};
		\node [style=label] (7) at (-4.5, -2) {$B \pcoprod C$};
		\node [style=label] (8) at (-6.25, -2) {$A$};
		\node [style=label] (9) at (-5.25, 2) {$D$};
	\end{pgfonlayer}
	\begin{pgfonlayer}{edgelayer}
		\draw (5.center) to (2.center);
		\draw (4.center) to (3.center);
		\draw (6.center) to (0.center);
	\end{pgfonlayer}
\end{tikzpicture}}

\leftrightarrow
\
\scalebox{0.8}{\begin{tikzpicture}
	\begin{pgfonlayer}{nodelayer}
		\node [style=none] (0) at (-1, 1.5) {};
		\node [style=none] (1) at (-0.5, -1.25) {};
		\node [style=none] (2) at (-0.5, -1.25) {};
		\node [style=none] (3) at (0, -0.75) {};
		\node [style=none] (4) at (0, -0.75) {};
		\node [style=none] (5) at (-0.5, -1.25) {};
		\node [style=none] (6) at (-1, -0.75) {};
		\node [style=none] (7) at (0, -0.75) {};
		\node [style=none] (8) at (1.5, -0.25) {};
		\node [style=none] (9) at (1.5, -1.5) {};
		\node [style=none] (10) at (0, -0.25) {};
		\node [style=none] (11) at (0.75, 0.5) {};
		\node [style=medium box] (12) at (0.75, -0) {$f$  };
		\node [style=none] (13) at (0.75, 1.5) {};
		\node [style=label] (14) at (-1, 2) {$A$};
		\node [style=label] (15) at (0.75, 2) {$D$};
		\node [style=label] (16) at (1.5, -2) {$B \pcoprod C$};
	\end{pgfonlayer}
	\begin{pgfonlayer}{edgelayer}
		\draw [style=none, bend right=45, looseness=1.25] (1.center) to (3.center);
		\draw [style=none, bend left=45, looseness=1.25] (5.center) to (6.center);
		\draw (10.center) to (7.center);
		\draw (8.center) to (9.center);
		\draw (11.center) to (13.center);
		\draw [style={arrow=.5}] (0.center) to (6.center);
	\end{pgfonlayer}
\end{tikzpicture}}

\]
to see that $A \otimes (B \pcoprod C)$ forms a phased coproduct of $A \otimes B$ and $A \otimes C$ with every phase of the form $\id{A} \otimes U$ for a phase $U$ of $B \pcoprod C$, as required. 

\ref{enum:phccompclosed}. By Theorem~\ref{thm:constr_is_monoidal} $\plusI{\ctb}$ is now a monoidal category and the functor $[-]$ is strict monoidal. 
Let $\obb{A} = A \pcoprod I$ be an object of $\plusI{\ctb}$, and $A^*$ be dual to $A$ in $\ctb$ via the state $\tinycup$ and effect $\tinycap$. For any object $\obb{A^*} = A^* \pcoprod I$ and morphisms $\mathbf{\eta}$, $\mathbf{\epsilon}$ in $\plusI{\ctb}$ with $[\mathbf{\eta}] = \tinycup$ and $[\mathbf{\epsilon}] = \tinycap$ we have 
\[
\left[
\scalebox{0.8}{\begin{tikzpicture}
	\begin{pgfonlayer}{nodelayer}
		\node [style=wide point] (0) at (-1.5, -1) {$\eta$ };
		\node [style=none] (1) at (-0.75, 0.75) {};
		\node [style=none] (2) at (-2.25, 0.5) {};
		\node [style=none] (3) at (-0.75, -0.75) {};
		\node [style=none] (4) at (-2.25, -0.75) {};
		\node [style=wide copoint] (5) at (-3, 0.75) {$\epsilon$};
		\node [style=none] (6) at (-3.75, 0.5) {};
		\node [style=none] (7) at (-3.75, -1) {};
		\node [style=label] (8) at (-3.75, -1.5) {$\obb{A}$};
		\node [style=label] (9) at (-0.75, 1.25) {$\obb{A}$};
	\end{pgfonlayer}
	\begin{pgfonlayer}{edgelayer}
		\draw (3.center) to (1.center);
		\draw (4.center) to (2.center);
		\draw [style=none] (6.center) to (7.center);
	\end{pgfonlayer}
\end{tikzpicture}}

\right]
=
\scalebox{0.8}{\begin{tikzpicture}
	\begin{pgfonlayer}{nodelayer}
		\node [style=none] (0) at (-3.75, 0.75) {};
		\node [style=none] (1) at (-3.75, -0.75) {};
		\node [style=label] (2) at (-3.75, -1.25) {$A$};
		\node [style=none] (3) at (-2.25, -0.5) {};
		\node [style=none] (4) at (-0.75, 1) {};
		\node [style=none] (5) at (-0.75, -0.5) {};
		\node [style=label] (6) at (-0.75, 1.5) {$A$};
		\node [style=none] (7) at (-2.25, 0.75) {};
		\node [style=none] (8) at (-3, 1.5) {};
		\node [style=none] (9) at (-1.5, -1.5) {};
		\node [style=none] (10) at (-2.25, -0.5) {};
		\node [style=none] (11) at (-2.25, -0.5) {};
		\node [style=none] (12) at (-0.75, -0.5) {};
		\node [style=none] (13) at (-0.75, -0.5) {};
		\node [style=none] (14) at (-1.5, -1.25) {};
		\node [style=none] (15) at (-2.25, -0.5) {};
		\node [style=none] (16) at (-1.5, -1.25) {};
		\node [style=none] (17) at (-3.75, 0.75) {};
		\node [style=none] (18) at (-3.75, 0.75) {};
		\node [style=none] (19) at (-2.25, 0.75) {};
		\node [style=none] (20) at (-2.25, 0.75) {};
		\node [style=none] (21) at (-3, 1.5) {};
		\node [style=none] (22) at (-3.75, 0.75) {};
		\node [style=none] (23) at (-3, 1.5) {};
	\end{pgfonlayer}
	\begin{pgfonlayer}{edgelayer}
		\draw [style=none, bend left=45, looseness=1.25] (14.center) to (11.center);
		\draw [style=none, bend right=45, looseness=1.25] (16.center) to (12.center);
		\draw [style=none, bend right=45, looseness=1.25] (21.center) to (18.center);
		\draw [style=none, bend left=45, looseness=1.25] (23.center) to (19.center);
		\draw [style={arrow=.5}] (7.center) to (3.center);
		\draw [style={arrow=.5}] (1.center) to (0.center);
		\draw [style={arrow=.5}] (5.center) to (4.center);
	\end{pgfonlayer}
\end{tikzpicture}}

=
\scalebox{0.8}{\begin{tikzpicture}
	\begin{pgfonlayer}{nodelayer}
		\node [style=none] (0) at (-0.75, -0.75) {};
		\node [style=label] (1) at (-0.75, -1.25) {$A$};
		\node [style=none] (2) at (-0.75, 1) {};
		\node [style=label] (3) at (-0.75, 1.5) {$A$};
	\end{pgfonlayer}
	\begin{pgfonlayer}{edgelayer}
		\draw [style=none] (0.center) to (2.center);
	\end{pgfonlayer}
\end{tikzpicture}}

\]
where the diagram inside $[-]$ is in $\plusI{\ctb}$. Similarly the other snake equation also holds. Then in $\plusI{\ctb}$ we have 
\[
\scalebox{0.8}{\begin{tikzpicture}
	\begin{pgfonlayer}{nodelayer}
		\node [style=wide point] (0) at (6.75, -0.75) {$\eta$ };
		\node [style=none] (1) at (6, 1) {};
		\node [style=none] (2) at (7.5, 0.75) {};
		\node [style=none] (3) at (6, -0.5) {};
		\node [style=none] (4) at (7.5, -0.5) {};
		\node [style=wide copoint] (5) at (8.25, 1) {$\epsilon$};
		\node [style=none] (6) at (9, 0.75) {};
		\node [style=none] (7) at (9, -0.75) {};
		\node [style=label] (8) at (9, -1.25) {$\obb{A^*}$};
		\node [style=label] (9) at (6, 1.5) {$\obb{A^*}$};
		\node [style=none] (10) at (10.25, -0) {$=$};
		\node [style=none] (11) at (11.5, -1) {};
		\node [style=none] (12) at (11.5, 1) {};
		\node [style=label] (13) at (11.5, 1.5) {$\obb{A^*}$};
		\node [style=label] (14) at (11.5, -1.5) {$\obb{A^*}$};
		\node [style=box] (15) at (11.5, -0) {$V$};
		\node [style=none] (16) at (1.25, 1) {};
		\node [style=none] (17) at (-1, 1) {};
		\node [style=label] (18) at (1.25, -1.5) {$\obb{A}$};
		\node [style=none] (19) at (-1, -0.5) {};
		\node [style=box] (20) at (1.25, -0) {$U$};
		\node [style=label] (21) at (-1, 1.5) {$\obb{A}$};
		\node [style=wide point] (22) at (-1.75, -0.75) {$\eta$ };
		\node [style=none] (23) at (1.25, -1) {};
		\node [style=none] (24) at (-2.5, -0.5) {};
		\node [style=label] (25) at (-4, -1.25) {$\obb{A}$};
		\node [style=wide copoint] (26) at (-3.25, 1) {$\epsilon$};
		\node [style=none] (27) at (-4, -0.75) {};
		\node [style=label] (28) at (1.25, 1.5) {$\obb{A}$};
		\node [style=none] (29) at (-2.5, 0.75) {};
		\node [style=none] (30) at (-4, 0.75) {};
		\node [style=none] (31) at (0, -0) {$=$};
	\end{pgfonlayer}
	\begin{pgfonlayer}{edgelayer}
		\draw (3.center) to (1.center);
		\draw (4.center) to (2.center);
		\draw [style=none] (6.center) to (7.center);
		\draw [style=none] (12.center) to (11.center);
		\draw (19.center) to (17.center);
		\draw (24.center) to (29.center);
		\draw [style=none] (30.center) to (27.center);
		\draw [style=none] (16.center) to (23.center);
	\end{pgfonlayer}
\end{tikzpicture}}

\]
for some phases $U,V$ in $\ctb$. Since $U$ and $V$ are invertible, setting  
\[
\scalebox{0.8}{\begin{tikzpicture}
	\begin{pgfonlayer}{nodelayer}
		\node [style=none] (0) at (0.25, 0.75) {};
		\node [style=none] (1) at (0.25, -1) {};
		\node [style=wide copoint] (2) at (1, 1) {$\epsilon$};
		\node [style=none] (3) at (1.75, 0.75) {};
		\node [style=none] (4) at (1.75, -1) {};
		\node [style=label] (5) at (1.75, -1.5) {$\obb{A^*}$};
		\node [style=label] (6) at (-2.5, -1.5) {$\obb{A^*}$};
		\node [style=box] (7) at (0.25, -0.25) {$U^{-1}$};
		\node [style=label] (8) at (0.25, -1.5) {$\obb{A}$};
		\node [style=none] (9) at (-2.5, -1) {};
		\node [style=label] (10) at (-4, -1.5) {$\obb{A}$};
		\node [style=wide copoint] (11) at (-3.25, 1) {$\epsilon'$};
		\node [style=none] (12) at (-4, -1) {};
		\node [style=none] (13) at (-2.5, 0.75) {};
		\node [style=none] (14) at (-4, 0.75) {};
		\node [style=none] (15) at (-1.25, -0) {$:=$};
	\end{pgfonlayer}
	\begin{pgfonlayer}{edgelayer}
		\draw (1.center) to (0.center);
		\draw [style=none] (3.center) to (4.center);
		\draw (9.center) to (13.center);
		\draw [style=none] (14.center) to (12.center);
	\end{pgfonlayer}
\end{tikzpicture}}

\]
one may check that $\eta$ and $\epsilon'$ form a dual pair in $\plusI{\ctb}$. 

\ref{enum:pharescalars}. Let $U \colon \obb{A} \to \obb{A}$ be a phase in $\ctb$. In $\plusI{\ctb}$ we have 
\[
\left[
\scalebox{0.8}{\begin{tikzpicture}
	\begin{pgfonlayer}{nodelayer}
		\node [style=none] (0) at (-3.75, 1.5) {};
		\node [style=none] (1) at (-3.75, -0.5) {};
		\node [style=box] (2) at (-3.75, 0.5) {$U$};
		\node [style=wide point] (3) at (-4.5, -0.75) {$\eta$ };
		\node [style=none] (4) at (-5.25, -0.5) {};
		\node [style=none] (5) at (-5.25, 1.5) {};
	\end{pgfonlayer}
	\begin{pgfonlayer}{edgelayer}
		\draw (1.center) to (0.center);
		\draw (4.center) to (5.center);
	\end{pgfonlayer}
\end{tikzpicture}}

\right]
=
\left[
\scalebox{0.8}{\begin{tikzpicture}
	\begin{pgfonlayer}{nodelayer}
		\node [style=none] (0) at (0, 1.5) {};
		\node [style=none] (1) at (0, -0.5) {};
		\node [style=none] (2) at (-1.5, -0.5) {};
		\node [style=wide point] (3) at (-0.75, -0.75) {$\eta$ };
		\node [style=none] (4) at (-1.5, 1.5) {};
	\end{pgfonlayer}
	\begin{pgfonlayer}{edgelayer}
		\draw (1.center) to (0.center);
		\draw (2.center) to (4.center);
	\end{pgfonlayer}
\end{tikzpicture}}

\right]
\qquad
\text{ so that }
\qquad
\scalebox{0.8}{\begin{tikzpicture}
	\begin{pgfonlayer}{nodelayer}
		\node [style=none] (0) at (14.75, 1.5) {};
		\node [style=none] (1) at (12.5, -0.5) {};
		\node [style=wide point] (2) at (11.75, -0.75) {$\eta$ };
		\node [style=none] (3) at (13.75, -0) {$=$};
		\node [style=none] (4) at (11, -0.5) {};
		\node [style=none] (5) at (11, 1.5) {};
		\node [style=scalar] (6) at (15.5, -0.75) {$u$};
		\node [style=none] (7) at (14.75, 0.75) {};
		\node [style=wide point] (8) at (15.5, 0.5) {$\eta$ };
		\node [style=none] (9) at (16.25, 0.75) {};
		\node [style=none] (10) at (12.5, 1.5) {};
		\node [style=none] (11) at (16.25, 1.5) {};
		\node [style=box] (12) at (12.5, 0.5) {$U$};
	\end{pgfonlayer}
	\begin{pgfonlayer}{edgelayer}
		\draw (1.center) to (10.center);
		\draw (4.center) to (5.center);
		\draw (9.center) to (11.center);
		\draw (7.center) to (0.center);
	\end{pgfonlayer}
\end{tikzpicture}}

\]
for some global phase $u$. But then $U = \id{\obb{A}} \cdot u$ since 
\[
\scalebox{0.8}{\begin{tikzpicture}
	\begin{pgfonlayer}{nodelayer}
		\node [style=scalar] (0) at (8.5, 0.75) {$u$};
		\node [style=none] (1) at (-1, 1.75) {};
		\node [style=none] (2) at (-1, -0.5) {};
		\node [style=box] (3) at (-1, 0.75) {$U$};
		\node [style=wide point] (4) at (-1.75, -0.75) {$\eta$ };
		\node [style=none] (5) at (-2.5, -0.5) {};
		\node [style=none] (6) at (-2.5, 1.5) {};
		\node [style=none] (7) at (6.75, 0.75) {$=$};
		\node [style=none] (8) at (7.75, -0.5) {};
		\node [style=none] (9) at (7.75, 2) {};
		\node [style=none] (10) at (6.75, 0.25) {};
		\node [style=box] (11) at (-7, 0.75) {$U$};
		\node [style=none] (12) at (-7, 2) {};
		\node [style=none] (13) at (-7, -0.5) {};
		\node [style=none] (14) at (-5.5, 0.75) {$=$};
		\node [style=none] (15) at (-4, 1.5) {};
		\node [style=none] (16) at (-2.5, 1.5) {};
		\node [style=wide copoint] (17) at (-3.25, 1.75) {$\epsilon'$ };
		\node [style=none] (18) at (-4, -0.5) {};
		\node [style=none] (19) at (4.5, 1.75) {};
		\node [style=none] (20) at (3, 1.5) {};
		\node [style=none] (21) at (4.5, -0.5) {};
		\node [style=wide copoint] (22) at (2.25, 1.75) {$\epsilon'$ };
		\node [style=wide point] (23) at (3.75, -0.75) {$\eta$ };
		\node [style=none] (24) at (1.5, 1.5) {};
		\node [style=none] (25) at (3, -0.5) {};
		\node [style=none] (26) at (1.5, -0.5) {};
		\node [style=none] (27) at (3, 1.5) {};
		\node [style=none] (28) at (0.25, 0.75) {$=$};
		\node [style=scalar] (29) at (5.25, 0.75) {$u$};
	\end{pgfonlayer}
	\begin{pgfonlayer}{edgelayer}
		\draw (2.center) to (1.center);
		\draw (5.center) to (6.center);
		\draw (8.center) to (9.center);
		\draw (13.center) to (12.center);
		\draw (15.center) to (18.center);
		\draw (21.center) to (19.center);
		\draw (25.center) to (20.center);
		\draw (24.center) to (26.center);
	\end{pgfonlayer}
\end{tikzpicture}}

\]
\end{proof}

\begin{corollary}
Let $\catC$ be a compact closed category with finite phased coproducts with transitive phases. Then $\catC$ has finite phased biproducts.
\end{corollary}

\begin{proof}
By Theorem~\ref{thm:getmoncoprod} and Lemma~\ref{lem:comp-closed}, $\plusI{\catC}$ is compact closed with distributive coproducts. But any compact closed category with finite coproducts has biproducts~\cite{houston2008finite}. Hence so does $\catC \simeq \plusI{\catC}_\quotP$ by Corollary~\ref{cor:biproducts}. 
\end{proof}

\begin{example}
$\FVecP$ and $\FHilbP$ each have finite phased biproducts. 
\end{example}

We leave open the question of whether compact closure automatically ensures that phases are transitive.

\section{Dagger Categories} \label{sec:daggers}

Our motivating examples $\Hilb$ and $\HilbP$ come with extra structure allowing us to identify global phases in the former category intrinsically, namely as those $z \in \mathbb{C}$ with $z^\dagger \cdot z = 1$. Recall that a \emph{dagger category}~\cite{selinger2008idempotents} is a category $\catC$ equipped with an involutive, identity-on-objects functor $(-)^\dagger \colon \catC^{\op} \to \catC$. 
In a dagger category, an \emph{isometry} is a morphism $f$ with $f^{\dagger} \circ f = \id{}$ and a \emph{unitary} further satisfies $f \circ f^{\dagger} = \id{}$~\cite[p.117]{selinger2008idempotents}. 
A \emph{dagger (braided, symmetric) monoidal} category is one with a dagger for which all coherence isomorphisms are unitary and $(f \otimes g)^{\dagger} = f^{\dagger} \otimes g^{\dagger}$ for all $f, g$.

\begin{definition} \label{def:daggerphbiprod}
In any dagger category with zero morphisms, a \emph{phased dagger biproduct} is a phased biproduct $A_1 \pbiprod \dots \pbiprod A_n$ for which $\pproj_i = \coproj_i^{\dagger}$ for all $i=1, \dots, n$.
\end{definition}

\noindent
A \emph{dagger biproduct}~\cite{selinger2008idempotents,vicary2011categorical} is then a phased dagger biproduct whose only phase is the identity. 

\begin{lemma}
Let $\catC$ be a dagger category with zero morphisms. Specifying a phased dagger biproduct $A \pbiprod B$ is equivalent to specifying a phased coproduct for which:
\begin{itemize}
\item 
$\pcoproj_A$ and $\pcoproj_B$ are isometries with $\pcoproj_A^\dagger \circ \pcoproj_B = 0$;
\item 
whenever $U$ is a phase so is $U^{\dagger}$.
\end{itemize}
\end{lemma}
\begin{proof}
The dagger sends phased coproducts to phased products and vice versa. 
The first point is a restatement of the equations of a biproduct, while the second is equivalent to the projections and coprojections then having the same phases.
\end{proof}

\begin{lemma} 
A dagger category has finite phased dagger biproducts iff it has a zero object and binary phased dagger biproducts.
\end{lemma}
\begin{proof}
We have seen that $(A \pbiprod B) \pbiprod C$ forms a phased biproduct of $A, B, C$ with coprojections $\pcoproj_{A \pbiprod B} \circ \pcoproj_A$, $\pcoproj_{A \pbiprod B} \circ \pcoproj_B$ and $\pcoproj_C$. But these are isometries whenever all of the $\pcoproj$ are. 
Similarly, we obtain phased dagger biproducts $A_1 \pbiprod \dots \pbiprod A_n$.
\end{proof}

Our motivating source of examples is the following.

\begin{lemma} \label{lem:dagger_ph_biprod_main}
Let $\catC$ be a dagger category with dagger biproducts and a choice of trivial isomorphisms which is transitive and closed under the dagger. Then $\quot{\catC}{\sim}$ is a dagger category with finite phased dagger biproducts.
\end{lemma}
\begin{proof}
This follows easily from Lemma~\ref{lem:getBiprod}, noting that thanks to our assumptions whenever $f \sim g$ then $f^\dagger \sim g^\dagger$ also, so that $\quot{\catC}{\sim}$ is indeed a dagger category.
\end{proof}

\begin{example}
$\Hilb$ is a dagger category with dagger biproducts, with the global phases $e^{i \theta}$ being precisely its unitary scalars. Thanks to this $\HilbP$ has finite phased dagger biproducts. 
\end{example}

We now desire versions of our results on $\plusI{-}$ for dagger categories. However, a problem arises from the fact that the canonical (non-unique) isomorphisms from Lemma~\ref{lem:isoms} or distributivity~\eqref{eq:distrib-isomorphism} need not be unitary as canonical isomorphisms in a dagger category should be. 

\begin{example} \label{ex:wocoherentdagphases}
Let $S$ be a commutative semi-ring with an involution $s \mapsto s^{\dagger}$. Write $\MatS$ for the dagger symmetric monoidal category with natural numbers as objects, and morphisms $M \colon n \to m$ being $S$-valued matrices, with $M^{\dagger}_{i,j} := (M_{j,i})^{\dagger}$ and $M \otimes N$ the Kronecker product of matrices. Then $\MatS$ has distributive dagger biproducts $n \biprod m := n + m$. Take as global phases $\mathbb{P}$ all $u \in S$ which are unitary, i.e.~$u^{\dagger} \cdot u = 1$.

Now suppose that $S$ has a unitary element of the form $s^{\dagger} \cdot s \neq 1$ for some $s \in S$; for example we may take $S = \mathbb{C}$ but with trivial involution $z^{\dagger} := z$ for all $z \in \mathbb{C}$, and choose $s = -1 = i^{\dagger} \cdot i$. Then the morphism $(1,0) \colon 1 \to 2$ in $\MatS$, together with either $(0,1)$ or $(0,s)$ makes the object $2$ a phased dagger biproduct $1 \pbiprod 1$ in $(\MatS)_\quotP$. But the endomorphism of $2$ in $\MatS$ with matrix 
\[
\begin{pmatrix}
1 & 0 \\ 0 & s
\end{pmatrix}
\] is not unitary, and nor is its induced morphism in $(\MatS)_\quotP$.
\end{example}

We can remedy this with an extra assumption about phased dagger biproducts. In a dagger category a morphism $f \colon A \to A$ is called \emph{positive} when $f = g^{\dagger} \circ g$ for some $g \colon A \to B$.

\begin{definition} \label{def:posfree}
We say that a dagger category with finite phased dagger biproducts has \deff{positive-free phases} when any phase $U$ on $A \pbiprod B$ which is positive has $U= \id{A \pbiprod B}$. 
\end{definition}

Equivalently, any morphism $f \colon A \pbiprod B \to C$ for which $f \circ \pcoproj_A$ and $f \circ \pcoproj_B$ are isometries with $\pcoproj_B^{\dagger} \circ f^\dagger \circ f \circ \pcoproj_A = 0$ is itself an isometry. It follows that positive phases of any phased dagger biproduct $A_1 \pbiprod \dots \pbiprod A_n$ are also trivial. In particular all phases and canonical and distributivity isomorphisms between finite phased dagger biproducts are unitary. 

For a category $\catC$ with phased dagger biproducts and a chosen object $I$, we define the category $\plusIdag{\catC}$ just like $\plusI{\catC}$ but with objects being phased dagger biproducts $\obb{A} = A \pbiprod I$. 

\begin{lemma} \label{lem:DagBiprodPlusI}
Let $\ctb$ be a category with finite phased dagger biproducts with positive-free phases and a phase generator $I$. Then $\plusIdag{\ctb}$ is a dagger category with finite dagger biproducts, and $[-] \colon \plusIdag{\ctb} \to \ctb$ preserves daggers. 
\end{lemma}
\begin{proof}
Let $\ctb$ be as above. One may check that any diagonal morphism $f \colon \obb{A} \to \obb{B}$ between phased dagger biproducts with $f \circ \pcoproj_i = \pcoproj_i \circ f_i$ has that $f^{\dagger} \colon \obb{B} \to \obb{A}$ is also diagonal with $f^{\dagger} \circ \pcoproj_i = \pcoproj_i \circ {f_i}^{\dagger}$. Hence $\plusIdag{\ctb}$ is a dagger category with the same dagger as $\ctb$, and $[-] \colon \ctb \to \plusIdag{\ctb}$ preserves daggers. 

Now any lifting $(\obb{A \pbiprod B}, \pcoproj_{\obb{A}}, \pcoproj_{\obb{B}})$ of a phased dagger biproduct in $\ctb$ is a biproduct in $\plusIdag{\ctb}$, just as in Lemma~\ref{lem:getBiprod}. Moreover each coprojection has that $[\pcoproj_{\obb{A}}^{\dagger} \circ \pcoproj_{\obb{A}}] = [\pcoproj_{\obb{A}}]^{\dagger} \circ [\pcoproj_{\obb{A}}] = \id{}$ and so  
$\pcoproj_{\obb{A}}^{\dagger} \circ \pcoproj_{\obb{A}}$ is a phase in $\ctb$, and hence by positive-freeness is the identity, making this a dagger biproduct.
\end{proof}

When $\ctb$ is a dagger monoidal category, in $\plusIdag{\ctb}$ we again set $\mathbb{P}$ to be the morphisms $\obb{I} \to \obb{I}$ in $\plusIdag{\ctb}$ which are phases in $\ctb$. We call a choice of global phases $\mathbb{P}$ on a dagger monoidal category \indef{positive-free} if whenever $p \cdot \id{A}$ is positive then it is equal to $\id{A}$, for any $p \in \mathbb{P}$ and object $A$.

\begin{corollary} \label{cor:daggerbiproducts}
There is a one-to-one correspondence, up to dagger monoidal equivalence, between dagger monoidal categories
\begin{itemize}
	\item $\ctb$ with distributive finite phased dagger biproducts with positive-free phases;
\item $\cta$ with distributive finite dagger biproducts and a positive-free choice of unitary global phases $\mathbb{P}$;
\end{itemize}
given by $\cta \mapsto \cta_{\quotP}$ and  $\ctb \mapsto \plusIdag{\ctb}$.
\end{corollary}
\begin{proof}
$\cta_\quotP$ has phased dagger biproducts by Lemma~\ref{lem:dagger_ph_biprod_main}, and from the description of phases in this category we see that they are positive-free iff $\mathbb{P}$ is positive-free in $\cta$. Conversely, for $\ctb$ as above apply Lemma~\ref{lem:DagBiprodPlusI} and Corollary~\ref{cor:biproducts}. Thanks to positive-freeness, every phase is a unitary and hence so are all elements of $\mathbb{P}$.

We define the monoidal structure on $\plusIdag{\ctb}$ just as on $\plusI{\ctb}$. By positive-freeness the morphisms $c_{\obb{A}, \obb{B}}$ are isometries, and this in turn ensures that $\plusIdag{\ctb}$ is dagger monoidal. To show this, we will use the observation that in any dagger category, if $i$ and $j$ are isometries and  the following commutes
\[
\begin{tikzcd}
A \rar{f} \dar[swap]{i} & B \dar{j} \rar{h} & A \dar{i} \\ 
C \rar[swap]{g} & D \rar[swap]{g^\dagger} & C  
\end{tikzcd}
\]
then $h = f^{\dagger}$, and whenever $g$ is unitary so is $f$. Applying this to the situation 
\[
\begin{tikzcd}
\obb{A} \tens \obb{B} \rar{f \tens g} \dar[swap]{c_{\obb{A}, \obb{B}}} & 
\obb{C} \tens \obb{D} \rar{f^{\dagger} \tens g^{\dagger}} \dar[swap]{c_{\obb{C}, \obb{D}}} & \obb{A} \tens \obb{B} \dar{c_{\obb{A}, \obb{B}}} \\ 
\obb{A} \otimes \obb{B} \rar[swap]{f \otimes g} & \obb{C} \otimes \obb{D} \rar[swap]{f^{\dagger} \otimes g^{\dagger}} & \obb{A} \otimes \obb{B} 
\end{tikzcd}
\]
using that $ f^{\dagger} \otimes g^{\dagger} = (f \otimes g)^{\dagger}$ we see that $ f^{\dagger} \tens g^{\dagger} = (f \tens g)^{\dagger}$ also. Similarly, applying this observation to the definition of $\aalpha$ shows that it is unitary. 

Now any morphism $\beta \colon \obb{I} \tens \obb{I} \to \obb{I}$ as in the proof of Theorem~\ref{thm:constr_is_monoidal} is unitary thanks to positive-freeness. 
The natural isomorphisms $\rrho$ in $\plusIdag{\ctb}$ satisfy $\rrho_{\obb{A}} \tens \id{\obb{I}} = (\id{\obb{A}} \tens \beta) \circ \aalpha_{\obb{A}, \obb{I}, \obb{I}}$. Since the latter is unitary, the dagger respects $\tens$, and the assignment $f \mapsto f \tens \id{\obb{I}}$ is injective, it follows that $\rrho_{\obb{A}}$ is unitary. Similarly, so is $\llambda_{\obb{A}}$.

Now since $[-]$ is dagger monoidal so is the equivalence $\ctb \simeq \plusIdag{\ctb}_\quotP$. Conversely, the equivalence $F \colon \cta \to \plusIdag{\cta_\quotP}$ preserves daggers by definition and is such that every object in $\plusIdag{\cta_\quotP}$ is unitarily isomorphic to $F(A)$, for some $A$, making it a dagger equivalence.
\end{proof}

It is also easy to see that whenever either of $\catC$ or $\catD$ is braided or symmetric dagger monoidal, so is the other and each of the above functors.

\begin{example} \label{example:Hilb-dag-construction}
$\Hilb$ is dagger symmetric monoidal with distributive finite dagger biproducts and its global phases $e^{i \theta}$ are positive-free. Hence there is a dagger monoidal equivalence 
\[
\Hilb \simeq \plusIdag{\HilbP}
\]
\end{example}

It follows from our next result that the phased biproducts in $\HilbP$ in fact satisfy a condition strengthening positive-freeness, which we now describe. Let us say that phased dagger biproducts have \indef{positive cancellation} when any positive diagonal endomorphisms $p, q$ of $A \pbiprod B$ with $p = q \circ U$ for some phase $U$ have $p = q$. 

\begin{lemma} \label{lem:pos-cancellation}
Let $\ctb$ be a dagger monoidal category with distributive finite phased dagger biproducts with positive-free phases. Then  positive cancellation holds in $\ctb$ iff in $\plusIdag{\ctb}$ we have
\begin{equation} \label{eq:pos-condition}
[p] = [q] \implies p = q
\end{equation}
for all positive morphisms $p, q$.
\end{lemma}

\begin{proof}
Let $p, q$ be positive in $\plusIdag{\ctb}$ with $[p] = [q]$. Then $p = q \circ U$ for some phase $U$, and so when positive cancellation holds we have $p = q$.

 Conversely, suppose $\plusIdag{\ctb}$ satisfies the above and that $p, q$ are positive diagonal endomorphisms of $A \pbiprod B$ in $\plusIdag{\ctb}$ with $[p] = [q] \circ [U]$ for some phase $[U]$ in $\ctb$. Then in $\plusIdag{\ctb}$ we have $[\pproj_A \circ p \circ \pcoproj_A] = [\pproj_A \circ q \circ \pcoproj_A]$ and so $\pproj_A \circ p \circ \pcoproj_A = \pproj_A \circ q \circ \pcoproj_A$, and similarly for $B$, giving $p = q$. Hence $\ctb$ has positive cancellation.
\end{proof}

\begin{example}
$\Hilb$ satisfies the condition \eqref{eq:pos-condition}. Indeed let $p, q$ be positive linear maps with $p = e^{i \cdot \theta} \circ q$. Then since $p = p^\dagger$, subtracting $p$ gives that either $p=q=0$ or $e^{i \cdot \theta} = \pm 1$. But any positive maps with $p + q = 0$ have $p = q = 0$ also. 
\end{example}

\subsection*{Dagger Compactness}

We now combine daggers with compact closure. Recall that a \emph{dagger dual} for an object $A$ in a dagger symmetric monoidal category is a dual $(A^*, \eta, \epsilon)$ with $\epsilon = \eta^{\dagger} \circ \sigma$. A \emph{dagger compact} category is a dagger symmetric one in which every object has a dagger dual. 

Although compactness of $\catC$ ensures compactness of $\plusIdag{\catC}$, to establish dagger compactness we make an extra assumption; it is an open question whether this is necessary. 

\begin{proposition}
Suppose that $\catC$ is dagger compact with phased dagger biproducts which are positive-free, and that in $\catC$ every object $A$ has a morphism $\psi \colon I \to A$ with 
\[
\scalebox{0.8}{\begin{tikzpicture}
	\begin{pgfonlayer}{nodelayer}
		\node [style=none] (0) at (0, 2.25) {};
		\node [style=none] (1) at (0, -0.25) {};
		\node [style=none] (2) at (1.25, 1) {$=$};
		\node [style=point] (3) at (-1, 0.25) {$\psi$};
		\node [style=copoint] (4) at (-1, 1.75) {$\psi^{\dagger}$};
		\node [style=none] (5) at (-1, 0.75) {};
		\node [style=none] (6) at (-1, 1.5) {};
		\node [style=none] (7) at (2.25, 2.25) {};
		\node [style=none] (8) at (2.25, -0.25) {};
		\node [style=label] (9) at (0, 2.75) {$A$};
		\node [style=label] (10) at (2.25, 2.75) {$A$};
		\node [style=label] (11) at (2.25, -0.75) {$A$};
		\node [style=label] (12) at (0, -0.75) {$A$};
		\node [style=label] (13) at (-1.5, 1) {$A$};
	\end{pgfonlayer}
	\begin{pgfonlayer}{edgelayer}
		\draw (1.center) to (0.center);
		\draw (6.center) to (3);
		\draw (8.center) to (7.center);
	\end{pgfonlayer}
\end{tikzpicture}}

\]
Then $\plusIdag{\catC}$ is dagger compact.
\end{proposition}

\begin{proof}
In $\ctb$, let $A$ and $A^*$ be dagger dual objects via the state $\tinycup$.  Let $\psi \colon I \to A$ be as above, and let $\phi \colon \obb{I} \to \obb{A}$
and $\eta \colon \obb{I} \to \obb{A^*} \tens \obb{A}$ in $\plusIdag{\ctb}$ with $[\phi] = \psi$ and $[\eta] = \tinycup$. Then applying $[-]$ we see that in $\plusIdag{\ctb}$ we have 
\[
\scalebox{0.8}{\begin{tikzpicture}
	\begin{pgfonlayer}{nodelayer}
		\node [style=label] (0) at (4.75, 3) {$\obb{A}$};
		\node [style=label] (1) at (4.75, -0.5) {$\obb{A}$};
		\node [style=none] (2) at (4.75, 2.5) {};
		\node [style=none] (3) at (4.75, -0) {};
		\node [style=none] (4) at (3.5, 1.25) {$=$};
		\node [style=none] (5) at (2.25, 0.25) {};
		\node [style=none] (6) at (0.75, 2.25) {};
		\node [style=none] (7) at (0.75, 2.25) {};
		\node [style=none] (8) at (-0.75, -0) {};
		\node [style=none] (9) at (-0.75, 2.25) {};
		\node [style=wide point] (10) at (1.5, -0) {$\bar{\eta}$ };
		\node [style=none] (11) at (0.75, 0.25) {};
		\node [style=wide copoint] (12) at (0, 2.5) {$\bar{\eta}^{\dagger}$ };
		\node [style=none] (13) at (2.25, 2.5) {};
		\node [style=label] (14) at (-0.75, -0.5) {$\obb{A}$};
		\node [style=label] (15) at (2.25, 3) {$\obb{A}$};
		\node [style=scalar] (16) at (5.5, 1.25) {$u$};
		\node [style=none] (17) at (-0.75, 0.75) {};
		\node [style=none] (18) at (0, 1.5) {};
		\node [style=none] (19) at (0, 1.5) {};
		\node [style=none] (20) at (0.75, 2.25) {};
		\node [style=none] (21) at (0.75, 0.75) {};
		\node [style=none] (22) at (0, 1.5) {};
		\node [style=none] (23) at (0, 1.5) {};
		\node [style=none] (24) at (-0.75, 2.25) {};
	\end{pgfonlayer}
	\begin{pgfonlayer}{edgelayer}
		\draw (3.center) to (2.center);
		\draw (5.center) to (13.center);
		\draw [bend left, looseness=1.00] (17.center) to (18.center);
		\draw [bend left, looseness=1.00] (20.center) to (19.center);
		\draw [bend right, looseness=1.00] (21.center) to (22.center);
		\draw [bend right, looseness=1.00] (24.center) to (23.center);
		\draw (17.center) to (8.center);
		\draw (21.center) to (11.center);
	\end{pgfonlayer}
\end{tikzpicture}}

\]
for some $u \in \mathbb{P}$ and 
\[
\left[
\scalebox{0.8}{\begin{tikzpicture}
	\begin{pgfonlayer}{nodelayer}
		\node [style=none] (0) at (0, 1.25) {};
		\node [style=none] (1) at (0, -1.25) {};
		\node [style=point] (2) at (-1, -0.75) {$\bar{\psi}$};
		\node [style=copoint] (3) at (-1, 0.75) {$\bar{\psi}^{\dagger}$};
		\node [style=none] (4) at (-1, -0.25) {};
		\node [style=none] (5) at (-1, 0.5) {};
		\node [style=label] (6) at (0, 1.75) {$\obb{A}$};
		\node [style=label] (7) at (0, -1.75) {$\obb{A}$};
		\node [style=label] (8) at (-1.5, -0) {$\obb{A}$};
	\end{pgfonlayer}
	\begin{pgfonlayer}{edgelayer}
		\draw (1.center) to (0.center);
		\draw (5.center) to (2);
	\end{pgfonlayer}
\end{tikzpicture}}

\right]
=
\left[
\scalebox{0.8}{\begin{tikzpicture}
	\begin{pgfonlayer}{nodelayer}
		\node [style=none] (0) at (2.25, 1.25) {};
		\node [style=none] (1) at (2.25, -1.25) {};
		\node [style=label] (2) at (2.25, 1.75) {$\obb{A}$};
		\node [style=label] (3) at (2.25, -1.75) {$\obb{A}$};
	\end{pgfonlayer}
	\begin{pgfonlayer}{edgelayer}
		\draw (1.center) to (0.center);
	\end{pgfonlayer}
\end{tikzpicture}}

\right]
\quad 
\text{    so that     }
\quad
\scalebox{0.8}{\begin{tikzpicture}
	\begin{pgfonlayer}{nodelayer}
		\node [style=label] (0) at (10.5, 1.75) {$\obb{A}$};
		\node [style=label] (1) at (10.5, -1.75) {$\obb{A}$};
		\node [style=label] (2) at (8.25, -1.75) {$\obb{A}$};
		\node [style=none] (3) at (10.5, 1.25) {};
		\node [style=none] (4) at (7.25, 0.75) {};
		\node [style=label] (5) at (6.75, -0) {$\obb{A}$};
		\node [style=none] (6) at (10.5, -1.25) {};
		\node [style=none] (7) at (7.25, -0.25) {};
		\node [style=label] (8) at (8.25, 1.75) {$\obb{A}$};
		\node [style=none] (9) at (8.25, 1.25) {};
		\node [style=copoint] (10) at (7.25, 0.75) {$\bar{\psi}^{\dagger}$};
		\node [style=point] (11) at (7.25, -0.75) {$\bar{\psi}$};
		\node [style=none] (12) at (8.25, -1.25) {};
		\node [style=none] (13) at (9.5, -0) {$=$};
	\end{pgfonlayer}
	\begin{pgfonlayer}{edgelayer}
		\draw (12.center) to (9.center);
		\draw (4.center) to (11);
		\draw (6.center) to (3.center);
	\end{pgfonlayer}
\end{tikzpicture}}

\]
by positive-freeness.
But then 
\[
\scalebox{0.8}{\begin{tikzpicture}
	\begin{pgfonlayer}{nodelayer}
		\node [style=label] (0) at (14, 3) {$\obb{A}$};
		\node [style=label] (1) at (14, -0.5) {$\obb{A}$};
		\node [style=label] (2) at (16.75, -0.5) {$\obb{A}$};
		\node [style=none] (3) at (14, 2.5) {};
		\node [style=none] (4) at (18, 2) {};
		\node [style=label] (5) at (17.5, 1.25) {$\obb{A}$};
		\node [style=none] (6) at (14, -0) {};
		\node [style=none] (7) at (18, 1) {};
		\node [style=label] (8) at (16.75, 3) {$\obb{A}$};
		\node [style=none] (9) at (16.75, 2.5) {};
		\node [style=copoint] (10) at (18, 2) {$\bar{\psi}^{\dagger}$};
		\node [style=point] (11) at (18, 0.5) {$\bar{\psi}$};
		\node [style=none] (12) at (16.75, -0) {};
		\node [style=none] (13) at (15.75, 1.25) {$=$};
		\node [style=none] (14) at (26, 0.5) {};
		\node [style=none] (15) at (24.5, 2.5) {};
		\node [style=none] (16) at (24.5, 2.5) {};
		\node [style=none] (17) at (23, 0.25) {};
		\node [style=none] (18) at (23, 2.5) {};
		\node [style=wide point] (19) at (25.25, 0.25) {$\bar{\eta}$ };
		\node [style=none] (20) at (24.5, 0.5) {};
		\node [style=wide copoint] (21) at (23.75, 2.75) {$\bar{\eta}^{\dagger}$ };
		\node [style=none] (22) at (26, 2.75) {};
		\node [style=scalar] (23) at (14.75, 1.25) {$u$};
		\node [style=none] (24) at (23.75, 1.75) {};
		\node [style=none] (25) at (23.75, 1.75) {};
		\node [style=none] (26) at (24.5, 2.5) {};
		\node [style=none] (27) at (24.5, 1) {};
		\node [style=none] (28) at (23.75, 1.75) {};
		\node [style=none] (29) at (23.75, 1.75) {};
		\node [style=none] (30) at (23, 2.5) {};
		\node [style=scalar] (31) at (19, 1.25) {$u$};
		\node [style=none] (32) at (20.25, 1.25) {$=$};
		\node [style=label] (33) at (21.75, 3) {$\obb{A}$};
		\node [style=copoint] (34) at (26, 2.75) {$\bar{\psi}^{\dagger}$};
		\node [style=none] (35) at (21.75, -0) {};
		\node [style=none] (36) at (21.75, 2.5) {};
		\node [style=point] (37) at (23, 0.25) {$\bar{\psi}$};
		\node [style=label] (38) at (21.75, -0.5) {$\obb{A}$};
		\node [style=none] (39) at (23, 1) {};
		\node [style=none] (40) at (27.5, 1.25) {$=$};
		\node [style=copoint] (41) at (32, 0.5) {$\bar{\psi}^{\dagger}$};
		\node [style=none] (42) at (32, 0.75) {};
		\node [style=none] (43) at (30.5, -0.5) {};
		\node [style=none] (44) at (30.5, 3.25) {};
		\node [style=none] (45) at (30.5, 3.25) {};
		\node [style=none] (46) at (32, 3.25) {};
		\node [style=none] (47) at (32, -0.5) {};
		\node [style=none] (48) at (32, 1.5) {};
		\node [style=none] (49) at (32, 3.25) {};
		\node [style=label] (50) at (29.25, -1.25) {$\obb{A}$};
		\node [style=wide point] (51) at (31.25, -0.75) {$\bar{\eta}$ };
		\node [style=point] (52) at (32, 2.25) {$\bar{\psi}$};
		\node [style=none] (53) at (29.25, 3.25) {};
		\node [style=wide copoint] (54) at (31.25, 3.5) {$\bar{\eta}^{\dagger}$ };
		\node [style=none] (55) at (29.25, -0.75) {};
		\node [style=none] (56) at (32, 3.25) {};
		\node [style=label] (57) at (29.25, 3.75) {$\obb{A}$};
		\node [style=none] (58) at (30.5, -0.5) {};
		\node [style=none] (59) at (30.5, 3.25) {};
	\end{pgfonlayer}
	\begin{pgfonlayer}{edgelayer}
		\draw (12.center) to (9.center);
		\draw (4.center) to (11);
		\draw (6.center) to (3.center);
		\draw (14.center) to (22.center);
		\draw [bend left, looseness=1.00] (26.center) to (25.center);
		\draw [bend right, looseness=1.00] (27.center) to (28.center);
		\draw [bend right, looseness=1.00] (30.center) to (29.center);
		\draw (27.center) to (20.center);
		\draw (35.center) to (36.center);
		\draw [bend left, looseness=1.00] (39.center) to (24.center);
		\draw (39.center) to (17.center);
		\draw (47.center) to (42.center);
		\draw (55.center) to (53.center);
		\draw (44.center) to (58.center);
		\draw (46.center) to (52);
	\end{pgfonlayer}
\end{tikzpicture}}

\]
Then by positive-freeness in $\plusIdag{\ctb}$ we have $\id{A} \cdot u = \id{A}$ , so that $\eta$ satisfies the first equation of a dagger dual. The second equation is shown identically. 
\end{proof}

\begin{example}
$\MatS$ is dagger-compact for any involutive commutative semi-ring $S$. In particular so are $\FHilb \simeq \Mat_{\mathbb{C}}$ and its quotient $\FHilbP$, which satisfies the above conditions, with $\FHilb \simeq \plusIdag{\FHilbP}$.
\end{example}

\bibliographystyle{alpha}
\bibliography{thesis-bib}

\appendix

\section{Universality} \label{app:universality}

Let us now demonstrate a universal property of the $\mathsf{GP}$ construction. 

We say that a category has \emph{chosen} finite coproducts when it comes with a choice of initial object and coproduct $A + B$ for each pair of objects $A, B$. Let $\MD$ be the category whose objects are monoidal categories $(\catC, \otimes)$ with chosen finite coproducts which are monic and distributive, with morphisms being strict monoidal functors $F \colon \catC \to \catD$ which preserve coproducts strictly, meaning that $F(0) = 0$ and $F(A + B) = F(A) + F(B)$ for all $A, B$. Define the category $\MDG$ similarly but with each object $\catC$ coming with a choice of global phases $\mathbb{P}$ and with $F \colon \catC \to \catD$ preserving them, i.e.~whenever $u$ is a global phase so is $F(u)$.

The forgetful functor $\MDG \to \MD$ has a left adjoint which chooses for each $\catC$ the trivial global phase group $\{\id{I}\}$, and the unit of the adjunction is an isomorphism, making this a coreflection.



Now let us say that a monoidal category has \emph{chosen} finite phased coproducts when it comes with a chosen initial object, phased coproduct $A \pcoprod B$ for each pair of objects $A, B$ and a choice of:
\begin{itemize}
\item isomorphisms  $(A \pcoprod B) \pcoprod C \simeq A \pcoprod (B \pcoprod C)$ and $(A \pcoprod B) \pcoprod C \simeq (B \pcoprod A) \pcoprod C$ preserving the $\pcoproj_A, \pcoproj_B, \pcoproj_C$ as in Proposition~\ref{prop:assoc}, for all objects $A, B, C$:
\item morphism $c_{A,B} \colon A \otimes B \pcoprod I \to (A \pcoprod I) \otimes (B \pcoprod I)$ satisfying~\eqref{eq:corner};
\item 
isomorphism $\beta \colon I \otimes I \pcoprod I \to I \pcoprod I$ as in the proof of Theorem~\ref{thm:constr_is_monoidal}
\end{itemize}

Define the category $\MDPh$ just like $\MD$ but replacing `coproducts' by `phased coproducts', now requiring morphisms $F$ to preserve the choices of $0$, $A \pcoprod B$ and all of these morphisms.

We consider the following slight adaptation of the $\mathsf{GP}$ construction. For an object $\catC$ of $\MDPh$ define $\GPa{\catC}$ to have the same objects as $\catC$, with morphisms $A \to B$ being diagonal morphisms $f \colon A \pcoprod I \to B \pcoprod I$ with $f \circ \pcoproj_I = \pcoproj_I$. Here we use the chosen phased coproduct $A \pcoprod I$ which comes with $\catC$ by definition, for each object $A$. Then $\GPa{\catC}$ is monoidal as before, using the chosen morphisms $c_{A,B}$ and $\beta$, now with monoidal unit $I$ and having $\otimes$ defined on objects just as in $\catC$. 

 \begin{theorem} \label{thm:equivalence of categories}
The assignments $\catC \mapsto \GPa{\catC}$ and $(\catD, \mathbb{P}) \mapsto \catD_{\mathbb{P}}$ extend to an equivalence of categories $\MDPh \simeq \MDG$.


\end{theorem}
\begin{proof}
First note that these assignments are well-defined on objects.
In one direction, $\GPa{\catC}$ comes with a (distributive) choice of coproducts with initial object being that in $\catC$ and $A + B$ being the chosen object $A \pcoprod B$ in $\catC$, via the canonical coprojections $\pcoproj_{A,I} \colon A \pcoprod I \to (A \pcoprod B) \pcoprod I$ and $\pcoproj_{B,I} \colon B \pcoprod I \to (A \pcoprod B) \pcoprod I$ coming from the chosen isomorphisms 
\[
(A \pcoprod B) \pcoprod I
\simeq
A \pcoprod (B \pcoprod I)
\qquad
(A \pcoprod B) \pcoprod I
\simeq
(B \pcoprod A) \pcoprod I
\simeq
B \pcoprod (A \pcoprod I)
\]
Conversely in $\catD_\mathbb{P}$ we can choose the initial object to be that in $\catD$ and choose $A \pcoprod B$ to be their coproduct $A + B$ in $\catD$. Then the remaining morphisms we need to choose all have a canonical (unique) choice $x$ with respect to the coproducts in $\catD$ and so we may then choose $[x]_{\mathbb{P}}$ in $\catD_{\mathbb{P}}$.

Now $(-)_\mathbb{P}$ is functorial; indeed any $F \colon \catC \to \catD$ in $\MDG$ preserves global phases and so satisfies $f \sim g \implies F(f) \sim F(g)$. Hence it restricts to a functor $F_\mathbb{P} \colon \catC_\mathbb{P} \to \catD_\mathbb{P}$, which is easily seen to again be strict monoidal. Since $F$ strictly preserves coproducts, $F_\mathbb{P}$ then strictly preserves phased coproducts and the above chosen morphisms. 

Conversely, any $F \colon \catC \to \catD$ in $\MDPh$ satisfies $F(A \pcoprod I) = F(A) \pcoprod F(I) = F(A) \pcoprod I$ and so also defines a functor $\GPa{\catC} \to \GPa{\catD}$. Since $F$ strictly preserves the above choices it then strictly preserves coproducts in $\MDPh$, where it is again strict monoidal. Hence $\GPa{-}$ is functorial also. 

Finally, the definition of $\GPa{-}$ makes our earlier equivalences now isomorphisms of categories $\catC \simeq \GPa{\catC}_\mathbb{P}$ and $\catD \simeq \GPa{\catD_\mathbb{P}}$. Moreover these isomorphisms are natural, yielding the above equivalence.
\end{proof}

Combining this result with the coreflections between $\MDG$ and $\MD$ gives the following.

\begin{corollary} \label{cor:coreflection}
The $\GPa{-}$ construction is left adjoint to the forgetful functor from $\MD$ to $\MDPh$:
%
\[
\begin{tikzcd}
\MD \arrow[rr,bend left = 20, "U"]
& 
\bot 
&
\MDPh 
\arrow[ll, bend left = 20,  "\GPa{-}"]
\end{tikzcd}
\]
Moreover this adjunction is a coreflection.
\end{corollary}

Specialising to braided monoidal categories and braided strict monoidal functors we define categories $\BMD$, $\BMDG$ and $\BMDPh$ similarly, and again we obtain the above adjunction. 
Additionally, the forgetful functor $U \colon \BMDG \to \BMD$ now has a right adjoint choosing for $\catC$ the global phase group $\Aut(I)$ of all invertible scalars in $\catC$, and this adjunction is a reflection. Using Theorem~\ref{thm:equivalence of categories} again yields the following.

\begin{corollary}
The $\GPa{-}$ construction has a right adjoint forming a reflection
\[
\begin{tikzcd}
\BMD \arrow[rr,bend right = 20, swap, "(-)_{\Aut(I)}"]
& 
\bot 
&
\BMDPh 
\arrow[ll, bend right = 20, swap, "\GPa{-}"]
\end{tikzcd}
\]
\end{corollary}

The correspondences of Corollaries~\ref{cor:biproducts} and~\ref{cor:daggerbiproducts} can be made functorial in a similar way. In future work it would be interesting to see if the construction of Theorem~\ref{thm:getmoncoprod} can also be made universal outside of the monoidal setting.

\end{document}